\documentclass[12pt]{amsart}
\usepackage{amsmath,latexsym,amsfonts,amssymb,amsthm}
\usepackage{geometry}
\usepackage{hyperref}
\usepackage{mathrsfs}
\usepackage{graphicx,color}
\usepackage{tikz-cd}
\usepackage{tikz}
\usetikzlibrary{positioning}
\usepackage{mathtools}
\usepackage[all,cmtip]{xy}

\newtheorem{prop}{Proposition}[section]
\newtheorem{thm}[prop]{Theorem}
\newtheorem{cor}[prop]{Corollary}

\newtheorem{lem}[prop]{Lemma}

\theoremstyle{definition}
\newtheorem{que}[prop]{Question}
\newtheorem{defn}[prop]{Definition}

\newtheorem{rem}[prop]{\it Remark}

\newtheorem*{claim*}{Claim}

\newcommand{\bP}{\mathbb{P}}
\newcommand{\bC}{\mathbb{C}}
\newcommand{\bR}{\mathbb{R}}
\newcommand{\bA}{\mathbb{A}}
\newcommand{\bQ}{\mathbb{Q}}
\newcommand{\bZ}{\mathbb{Z}}
\newcommand{\bN}{\mathbb{N}}

\newcommand{\bG}{\mathbb{G}}

\newcommand{\tX}{\widetilde{X}}

\newcommand{\tD}{\widetilde{D}}

\newcommand{\cX}{\mathcal{X}}

\newcommand{\cO}{\mathcal{O}}
\newcommand{\cL}{\mathcal{L}}

\newcommand{\cM}{\mathcal{M}}
\newcommand{\cF}{\mathcal{F}}

\newcommand{\cN}{\mathcal{N}}
\newcommand{\cS}{\mathcal{S}}
\newcommand{\cT}{\mathcal{T}}
\newcommand{\cH}{\mathcal{H}}

\newcommand{\fa}{\mathfrak{a}}

\newcommand{\rd}{\mathrm{d}}

\newcommand{\Spec}{\mathrm{Spec}}
\newcommand{\Supp}{\mathrm{Supp}}

\newcommand{\mult}{\mathrm{mult}}
\newcommand{\lct}{\mathrm{lct}}

\newcommand{\vol}{\mathrm{vol}}
\newcommand{\ord}{\mathrm{ord}}

\newcommand{\md}{\frac{\mult_x}{\deg}}
\newcommand{\reg}{\mathrm{reg}}
\newcommand{\Diff}{\mathrm{Diff}}
\newcommand{\codim}{\mathrm{codim}}
\newcommand{\Ext}{\mathrm{Ext}}
\newcommand{\Sym}{\mathrm{Sym}}
\newcommand{\Proj}{\mathrm{Proj}}
\newcommand{\Nklt}{\mathrm{Nklt}}
\newcommand{\Bs}{\mathrm{Bs}}
\newcommand{\Aut}{\mathrm{Aut}}
\newcommand{\Fut}{\mathrm{Fut}}
\newcommand{\Exc}{\mathrm{Exc}}
\newcommand{\hvol}{\widehat{\mathrm{vol}}}
\newcommand{\Val}{\mathrm{Val}}
\newcommand{\gr}{\mathrm{gr}}
\newcommand{\cV}{\mathcal{V}}
\newcommand{\ocX}{\overline{\cX}}
\newcommand{\ocV}{\overline{\cV}}
\newcommand{\ocL}{\overline{\cL}}

\newcommand{\tc}{\mathrm{tc}}
\newcommand{\cG}{\mathcal{G}}
\newcommand{\coker}{\mathrm{coker}}
\newcommand{\Rees}{\mathrm{Rees}}
\newcommand{\initl}{\mathrm{in}}

\numberwithin{equation}{section}

\title{On the sharpness of Tian's criterion for K-stability}

\author{Yuchen Liu}
\address{Department of Mathematics, Yale University, New Haven, CT 06511, USA.}
\email{yuchen.liu@yale.edu}

\author{Ziquan Zhuang}
\address{Department of Mathematics, Massachusetts Institute of Technology, Cambridge, MA 02139, USA.}
\email{ziquan@mit.edu}

\date{\today}

\begin{document}

\maketitle

\begin{abstract}
    Tian's criterion for K-stability states that a Fano variety of dimension $n$ whose alpha invariant is greater than $\frac{n}{n+1}$ is K-stable. We show that this criterion is sharp by constructing $n$-dimensional singular Fano varieties with alpha invariants $\frac{n}{n+1}$ that are not K-polystable for sufficiently large $n$. We also construct K-unstable Fano varieties with alpha invariants $\frac{n-1}{n}$.
\end{abstract}

\section{Introduction}

Given a complex Fano manifold $X$, Tian \cite{Tia87} introduced the \emph{$\alpha$-invariant} $\alpha(X)$ of $X$ which measures the integrability of exponentials of plurisubharmonic functions. In the same paper, Tian proved that $X$ admits a  K\"ahler-Einstein metric if $\alpha(X)>\frac{n}{n+1}$ where $n=\dim X$, known as Tian's criterion. 
In this article, we will use the algebraic interpretation of $\alpha$-invariants in terms of singularities of pairs due to Demailly \cite[Appendix A]{CS08} (see also \cite[Appendix A]{Shi10}). 

\begin{defn}[\cite{Tia87, CS08}]
Let $X$ be a $\bQ$-Fano variety, i.e. $X$ is a normal projective variety over $\bC$, $-K_X$ is $\bQ$-Cartier and ample, and $X$ has klt singularities.
The $\alpha$-invariant $\alpha(X)$ is defined as
\[
\alpha(X):=\inf\{\lct(X;D)\mid D
\textrm{ is an effective $\bQ$-divisor and }D\sim_{\bQ}-K_X\}.
\]
\end{defn}

Tian's criterion \cite{Tia87} has been generalized
and complemented by Demailly-Koll\'ar \cite{DK01}, Odaka-Sano \cite{OS-alpha} and Fujita \cite{Fuj16}. The following theorem summarizes their results in the language of K-stability.

\begin{thm}[\cite{Tia87, DK01, OS-alpha, Fuj16}]\label{thm:tian}
Let $X$ be an $n$-dimensional $\bQ$-Fano variety. 
Then 
\begin{enumerate}
    \item $X$ is K-stable if either $\alpha(X)>\frac{n}{n+1}$ or $\alpha(X)= \frac{n}{n+1}$, $n\ge 2$ and $X$ is smooth;
    \item $X$ is K-semistable if $\alpha(X)\geq \frac{n}{n+1}$.
\end{enumerate}
\end{thm}

The purpose of this article is to study the sharpness of the assumptions in Theorem \ref{thm:tian}. Our main result goes as follows.

\begin{thm}\label{thm:main}
For any $n\gg 0$ (more precisely, $n=4$ or $n\geq 7$ for $X$, and $n\geq 14$ for $Y$), there exist a hypersurface $X\subset\bP^{n+1}$ of degree $n+1$, and a complete intersection $Y\subset\bP^{n+2}$ of a hyperquadric and a degree $n$ hypersurface, such that the following properties hold:
\begin{enumerate}
    \item Both $X$ and $Y$ are Gorenstein canonical with only one singular point;
    \item $\alpha(X)=\frac{n}{n+1}$ and $X$ is not K-polystable;
    \item $\alpha(Y)=\frac{n-1}{n}$ and $Y$ is K-unstable.
\end{enumerate}
\end{thm}

These examples show that the assumptions of Theorem \ref{thm:tian} are almost sharp. More precisely, in Theorem \ref{thm:tian}.1 the smoothness assumption cannot be removed, and the lower bound of $\alpha(X)$ is Theorem \ref{thm:tian}.2 cannot be smaller than $\frac{n-1}{n}$. We remark here that using different methods, K. Fujita constructed a log Fano hyperplane arrangement $(X,\Delta)$ that is not K-polystable and $\alpha(X,\Delta)=\frac{n}{n+1}$ (see \cite[Section 9]{Fuj17}).

Our construction is motivated by the work \cite{Fuj16}. It follows from the argument of \cite{Fuj16} that if $X$ is a $\bQ$-Fano variety of dimension $n$ such that $\alpha(X)=\frac{n}{n+1}$ but $X$ is not K-stable, then $X$ has a weakly exceptional singularity whose corresponding Koll\'ar component has log discrepancy $n$. Candidates of such singularities have been proposed by Kudryavtsev \cite{K-plt-blowup} and Prokhorov \cite{P-plt-blowup}, although both of their constructions contain a gap (see Remark \ref{rem:gap}). Therefore, a large part of our argument is devoted to rectifying their construction, which eventually reduces to estimating the global log canonical thresholds of general hypersurfaces or complete intersections and relies heavily on Pukhlikov's technique of hypertangent divisors (see e.g. \cite[\S 3]{Puk-book}). As a consequence of this analysis, we also prove the following:

\begin{thm} \label{thm:lct-general-hyp}
Fix $\epsilon>0$. Then for $n\gg0$, we have $\lct(X;|H|_\bQ)=1$ where $X\subseteq\bP^{n+1}$ is a general hypersurface of degree $d\le (2-\epsilon)n$ and $H$ is the hyperplane class.
\end{thm}

If $X$ is a non-K-stable $\bQ$-Fano variety of dimension $n$ with $\alpha(X)=\frac{n}{n+1}$, then Fujita's characterization \cite{Fuj16} also implies that $X$ specially degenerates to a K-semistable $\bQ$-Fano variety $X_0$ with $\alpha(X_0)=\frac{1}{n+1}$ (see Corollary \ref{cor:polystable degeneration}). Indeed $X_0$ has the smallest $\alpha$-invariant among all K-semistable $\bQ$-Fano varieties by \cite[Theorem 3.5]{FO16}. These $\bQ$-Fano varieties have been studied by C. Jiang \cite{Jia17} where he showed that $\bP^n$ is the only K-semistable Fano manifold with the smallest $\alpha$-invariant. We provide a full characterization of such $\bQ$-Fano varieties to complement Jiang's results (see Theorem \ref{thm:jiang}).

The following result is inspired by the work of Blum and Xu in \cite{BX18} where it is shown that the moduli functor of uniformly K-stable $\bQ$-Fano varieties with fixed dimension and volume is represented by a separated Deligne-Mumford stack of finite type.

\begin{thm}\label{thm:alpha-moduli}
Let $n$ be a positive integer and $V$ be a positive rational number. Then the moduli functor of $\bQ$-Fano varieties $X$ satisfying $\alpha(X)>\frac{1}{2}$ of dimension $n$ and volume $V$ is represented by a separated Deligne-Mumford stack of finite type, which has a coarse moduli space that is a separated algebraic space. In particular, $\Aut(X)$ is finite for a $\bQ$-Fano variety $X$ satisfying $\alpha(X)>\frac{1}{2}$.
\end{thm}

We ask the following question about the sharpness of Tian's criterion and Jiang's conjecture \cite[Conjecture 1.6]{Jia17}.

\begin{que}
Let $n=\dim X\geq 2$ be an integer. 
\begin{enumerate}
    \item Does there exist a K-semistable $\bQ$-Fano variety $X$ such that $\frac{1}{n+1}<\alpha(X)<\frac{1}{n}$?
    \item Does there exist a K-unstable $\bQ$-Fano variety $X$ such that $\frac{n-1}{n}<\alpha(X)<\frac{n}{n+1}$?
\end{enumerate}
\end{que}

This paper is organized as follows. In Section \ref{sec:prelim} we collect some preliminary materials on K-stability, weakly exceptional singularities, Koll\'ar components, and orbifold cones. In Section \ref{sec:fuj} we recall Fujita's characterization on non-K-stable $\bQ$-Fano varieties $X$ of dimension $n$ with $\alpha(X)=\frac{n}{n+1}$.
In Section \ref{sec:fuj}, we prove Theorem \ref{thm:alpha-moduli}. We also provide a generalization of  \cite[Theorem 1.2]{Jia17} in Theorem \ref{thm:jiang}. In Section \ref{sec:lct-hyp} we estimate the global log canonical thresholds of general hypersurfaces or complete intersections based on Pukhlikov's work \cite{Puk-product, Puk-book, Puk-cpi}. Finally in Section \ref{sec:construction} we present our main constructions to prove Theorem \ref{thm:main}.

\subsection*{Acknowledgement}
This project was initiated to answer a question of Kento Fujita originated from his paper \cite{Fuj16}. The authors would like to thank Kento Fujita for inspiration and fruitful discussions. The authors would like to thank Harold Blum, Ivan Cheltsov, Christopher Hacon, Chen Jiang, J\'anos Koll\'ar, Chi Li, Gang Tian, and Chenyang Xu for helpful discussions and comments. They are also grateful to the anonymous referees for many helpful suggestions and comments. The first author would like to thank Sam Payne for his support during the Fall 2018 semester, constant encouragement, and helpful comments. The second author would also like to thank his advisor J\'anos Koll\'ar for constant support and encouragement. This material is based upon work supported by the National Science Foundation under Grant No. DMS-1440140 while both authors were in residence at the Mathematical Sciences Research Institute in Berkeley, California, during the Spring 2019 semester.

\section{Preliminary}\label{sec:prelim}

\subsection{Notation and conventions}

We work over the complex numbers. We follow the terminologies in \cite{km98}.  A projective variety $X$ is $\bQ$-\emph{Fano} if $X$ has klt singularities and $-K_X$ is ample. A pair $(X,\Delta)$ is \emph{log Fano} if $X$ is projective, $-K_X-\Delta$ is $\bQ$-Cartier ample and $(X,\Delta)$ is klt. If $(X,\Delta)$ is a klt pair, $X$ is projective and $H$ is an ample $\bQ$-divisor on $X$, then the \emph{global log canonical threshold} $\lct(X,\Delta;|H|_\bQ)$ is defined as the largest $t>0$ such that $(X,\Delta+tD)$ is lc for every effective $\bQ$-divisor $D\sim_\bQ H$. If $\Delta=0$, we simply write $\lct(X;|H|_\bQ)$. The \emph{$\alpha$-invariant} of a log Fano pair $(X,\Delta)$ is defined as $\alpha(X,\Delta):=\lct(X,\Delta;|-K_X-\Delta|_{\bQ})$.

\subsection{K-stability}

\begin{defn}[\cite{Tia97, Don02, LX14, Li15, OS15}]
Let $(X,\Delta)$ be an $n$-dimensional log Fano pair. Let $L$ be an ample line bundle on $X$ such that $L\sim_{\bQ}-l (K_X+\Delta)$ for some $l\in\bQ_{>0}$.
\begin{enumerate}
\item A \emph{normal test configuration} $(\cX,\Delta_{\tc};\cL)/\bA^1$ of $(X,\Delta;L)$ consists of the following data:
\begin{itemize}
 \item a normal variety $\cX$, an effective $\bQ$-divisor $\Delta_{\tc}$ on $\cX$, together with a flat projective morphism $\pi:(\cX,\Supp(\Delta_{\tc}))\to \bA^1$;
 \item a $\pi$-ample line bundle $\cL$ on $\cX$;
 \item a $\bG_m$-action on $(\cX,\Delta_{\tc};\cL)$ such that $\pi$ is $\bG_m$-equivariant with respect to the standard action of $\bG_m$ on $\bA^1$ via multiplication;
 \item $(\cX\setminus\cX_0,\Delta_{\tc}|_{\cX\setminus\cX_0};\cL|_{\cX\setminus\cX_0})$
 is $\bG_m$-equivariantly isomorphic to $(X,\Delta;L)\times(\bA^1\setminus\{0\})$.
\end{itemize}

A normal test configuration is called a \emph{product} test configuration if 
\[
(\cX,\Delta_{\tc};\cL)\cong(X\times\bA^1,\Delta\times\bA^1;\mathrm{pr}_1^* L\otimes\mathrm{pr}_2^*\cO_{\bA^1}(k\cdot 0))
\]
for some $k\in\bZ$. A product test configuration is called a \emph{trivial} test configuration if the above isomorphism is $\bG_m$-equivariant with respect to the trivial $\bG_m$-action on $X$ and the standard $\bG_m$-action on $\bA^1$ via multiplication.

A normal test configuration $(\cX,\Delta_{\tc};\cL)$ is called a \emph{special test configuration} if $\cL\sim_{\bQ}-l (K_{\cX/\bA^1}+\Delta_{\tc})$ and $(\cX,\cX_0+\Delta_{\tc})$ is plt. In this case, we say that $(X,\Delta)$ \emph{specially degenerates to} $(\cX_0,\Delta_{\tc,0})$ which is necessarily a log Fano pair.

\item Assume $\pi:(\cX,\Delta_{\tc};\cL)\to \bA^1$ is a normal test configuration of 
$(X,\Delta;L)$. Let $\bar{\pi}: (\ocX,\overline{\Delta}_{\tc};\ocL)\to\bP^1$ be the natural $\bG_m$-equivariant compactification of $\pi$. The \emph{generalized Futaki invariant} of $(\cX,\Delta_{\tc};\cL)$ is defined by the intersection formula
\[
\Fut(\cX,\Delta_{\tc};\cL):=\frac{1}{(-K_X-\Delta)^n}\left(\frac{n}{n+1}\cdot\frac{(\ocL^{n+1})}{l^{n+1}}+\frac{(\ocL^n\cdot (K_{\ocX/\bP^1}+\overline{\Delta}_{\tc}))}{l^n}\right).
\]
\item \begin{itemize}
 \item The log Fano pair $(X,\Delta)$ is said to be \emph{K-semistable} if $\Fut(\cX,\Delta_{\tc};\cL)\geq 0$ for any normal test configuration $(\cX,\Delta_{\tc};\cL)/\bA^1$ and any $l\in\bQ_{>0}$ such that $L$ is Cartier. 
 \item The log Fano pair $(X,\Delta)$ is said to be \emph{K-stable} if it is K-semistable and $\Fut(\cX,\Delta_{\tc};\cL)=0$ for a normal test configuration $(\cX,\Delta_{\tc};\cL)/\bA^1$ if and only if it is a trivial test configuration.
 \item The log Fano pair $(X,\Delta)$ is said to be \emph{K-polystable} if it is K-semistable and $\Fut(\cX,\Delta_{\tc};\cL)=0$ for a normal test configuration $(\cX,\Delta_{\tc};\cL)/\bA^1$ if and only if it is a product test configuration.
 \item The log Fano pair $(X,\Delta)$ is said to be \emph{strictly K-semistable} if it is K-semistable but not K-polystable.
 \item The log Fano pair $(X,\Delta)$ is said to be \emph{K-unstable} if it is not K-semistable.
\end{itemize}
\end{enumerate}
\end{defn}

 By the work of Li and Xu \cite{LX14}, to test K-(poly/semi)stability of a log Fano pair $(X,\Delta)$ it suffices to test on special test configurations (see \cite[Section 6]{Fujita-valuative-criterion} for precise statements).
\medskip 

Next we recall the valuative criterion of K-stability due to K. Fujita \cite{Fuj17} and C. Li \cite{Li-minimize}.

\begin{defn}[{\cite[Definition 1.1 and 1.3]{Fujita-valuative-criterion}}] \label{defn:threshold and beta}
Let $X$ be a $\bQ$-Fano variety of dimension $n$. Let $F$ be a prime divisor over $X$, i.e., there exists a projective birational morphism $\pi: Y\to X$ with $Y$ normal such that $F$ is a prime divisor on $Y$.
    \begin{enumerate}
        \item For any $t\ge 0$, we define $\vol_X(-K_X-tF):=\vol_Y(-\pi^*K_X-tF)$.
        \item The \emph{pseudo-effective threshold} $\tau(F)$ of $F$ with respect to $-K_X$ is defined as
        \[\tau(F):=\sup\{\tau>0\,|\, \vol_X(-K_X-\tau F)>0\}.\]
        \item Let $A_X(F)$ be the log discrepancy of $F$ with respect to $X$. We set
        \[\beta(F):=A_X(F)\cdot((-K_X)^n)-\int_0^{\tau(F)}\vol_X(-K_X-tF)\mathrm{d}t.\]
        \item The prime divisor $F$ over $X$ is said to be \emph{dreamy} if the $\bZ_{\geq 0}^2$-graded algebra \[\bigoplus_{k,j\in\bZ_{\geq 0}}H^0(X,\cO_X(-klK_X-jF))\] is finitely generated for some $l\in\bZ_{>0}$ with $-lK_X$ Cartier. Note that this definition does not depend on the choice of $l$. 
    \end{enumerate}
\end{defn}

The following theorem summarizes results from \cite[Theorems 1.3 and 1.4]{Fujita-valuative-criterion}, \cite[Theorem 3.7]{Li-minimize}, and \cite[Corollary 4.3]{BX18}.

\begin{thm}[\cite{Fujita-valuative-criterion, Li-minimize, BX18}] \label{thm:beta-criterion}
Let $X$ be a $\bQ$-Fano variety. Then the following are equivalent:
\begin{enumerate}
    \item $X$ is K-stable $($resp.\ K-semistable$)$;
    \item $\beta(F)>0$ $($resp.\ $\beta(F)\geq 0)$ holds for every prime divisor $F$ over $X$;
    \item $\beta(F)>0$ $($resp.\ $\beta(F)\geq 0)$ holds for every dreamy prime divisor $F$ over $X$. 
\end{enumerate}
\end{thm}

\subsection{Weakly exceptional singularities and Koll\'ar components}

\begin{defn}
Let $x\in X$ be a klt singularity. We say that a proper birational morphism $\sigma:Y\to X$ provides a \emph{Koll\'ar component} $F$, if $\sigma$ is an isomorphism over $X\setminus\{x\}$ and $\sigma^{-1}(x)$ is a prime divisor $F$ on $Y$ such that $(Y,F)$ is plt and $-F$ is $\bQ$-Cartier $\sigma$-ample. We denote by $\Delta_F$ the different $\bQ$-divisor on $F$, i.e. $K_F+\Delta_F=(K_Y+F)|_F$.
\end{defn}

\begin{defn}
Let $X$ be a $\bQ$-Fano variety. Let $x\in X$ be a closed point. Suppose $\sigma:Y\to X$ provides a Koll\'ar component $F$ over $x\in X$. Then the \emph{Seshadri constant} $\epsilon(F)$ of $F$ with respect to $-K_X$ is defined as
\[
\epsilon(F):=\sup\{\epsilon>0\,|\, \sigma^*(-K_X)-\epsilon F\textrm{ is ample}\}.
\]
\end{defn}

The notion of weakly exceptional singularities in the sense of Shokurov is crucial in our construction.

\begin{defn}[{\cite[Definition 4.1]{Pro00}}]
A klt singularity $x\in X$ is said to be \emph{weakly exceptional} if there exists only one Koll\'ar component over it.
\end{defn}

The following criterion connects weakly exceptional singularities with $\alpha$-invariant of Koll\'ar components. Its present form first appeared in \cite[Theorem 3.10]{CS11} which essentially follows from {\cite[Theorem 4.3]{Pro00} and \cite[Theorem 2.1]{K-plt-blowup}}.

\begin{thm}[{\cite[Theorem 3.10]{CS11}}] \label{thm:alpha and weakly exceptional}
A klt singularity $x\in X$ is weakly exceptional if and only if there exists a Koll\'ar component $F$ over $x\in X$ satisfying $\alpha(F,\Delta_F)\geq 1$. 
\end{thm}

\subsection{Orbifold cones}

\begin{defn}
Let $V$ be a normal projective variety. Let $M$ be an ample $\bQ$-Cartier $\bQ$-divisor on $V$. 
\begin{enumerate}
    \item The \emph{affine orbifold cone} $C_a(V,M)$ is defined as
    \[
    C_a(V,M):=\Spec\bigoplus_{m=0}^\infty H^0(V,\cO_{V}(\lfloor mM \rfloor)).
    \]
    \item The \emph{projective orbifold cone} $C_p(V,M)$ is defined as
    \[
    C_p(V,M):=\Proj\bigoplus_{m=0}^\infty\bigoplus_{i=0}^\infty H^0(V,\cO_{V}(\lfloor mM \rfloor)\cdot s^i,
    \]
    where the grading of $H^0(V,\cO_{V}(\lfloor mM \rfloor))$ and $s$ are $m$ and $1$, respectively.
\end{enumerate}
\end{defn}

For the projective orbifold cone $C_p(V,M)$, we denote the $\bQ$-Cartier divisor corresponding to $(s=0)$ by $V_\infty$. If $\{M\}=\sum_{i=1}^k \frac{a_i}{b_i}M_i$ for some prime divisors $M_i$ on $V$ and $0<a_i<b_i$ coprime integers, we denote 
$\Delta_M:=\sum_{i=1}^k\frac{b_i-1}{b_i}M_i$. The pair $(V_\infty, \Diff_{V_\infty}(0))$ obtained by taking adjunction of the pair $(C_p(V,M), V_\infty)$ is isomorphic to $(V, \Delta_M)$. 

Let us illustrate some basic properties of orbifold cones from \cite{Kol04} and \cite[Section 3.3.1]{LL19}.
\begin{prop}
Let $V$ be a normal projective variety. Let $M$ be an ample $\bQ$-Cartier divisor on $V$.
\begin{enumerate}
    \item Both $C_a(V,M)$ and $C_p(V,M)$ are normal.
    \item The following conditions are equivalent:
    \begin{enumerate}
        \item $C_a(V,M)$ is klt;
        \item $(C_p(V,M),V_\infty)$ is plt;
        \item $(V,\Delta_M)$ is a log Fano pair, and $M\sim_{\bQ}-r^{-1}(K_V+\Delta_M)$ for some $r\in\bQ_{>0}$.
    \end{enumerate}
    \item If the conditions in $(2)$ are satisfied, then $K_{C_p(V,M)}\sim_{\bQ}-(1+r)V_\infty$.
\end{enumerate}

\end{prop}

The following proposition has appeared in \cite[Section 2.4]{LX16}. Here we provide a proof without using stack constructions.

\begin{prop}\label{prop:kcdeg}
Let $x\in (X,\Delta)$ be a klt singularity. Suppose $\sigma:Y\to X$ provides a Koll\'ar component $F$ over $x$.
Denote by $\fa_j:=\{f\in\cO_{X,x}\mid \ord_F(f)\geq j\}$ for $j\in\bZ$. Then we have an isomorphism of graded rings
\[
 \bigoplus_{j=0}^\infty \fa_j/\fa_{j+1}\cong \bigoplus_{j=0}^\infty H^0(F,\cO_F(\lfloor -jF|_F \rfloor)),
 \]
where $F|_F$ is the $\bQ$-divisor class on $F$ defined in \cite[Definition A.4]{HLS}.
\end{prop}

\begin{proof}
Since $X$ is normal, we have $\fa_j=\sigma_*\cO_Y(-jF)$. Thus we have the following exact sequence
\[
 0\to\fa_{j+1}\to\fa_{j}\to \sigma_*(\cO_Y(-jF)/\cO_Y(-(j+1)F))\to R^1\sigma_*\cO_Y(-(j+1)F).
\]
Since  
\[
-(j+1)F-(K_Y+\sigma_*^{-1}\Delta)=-\sigma^*(K_X+\Delta)-(j+A_{(X,\Delta)}(\ord_F))F,
\]
the divisor $-(j+1)F-(K_Y+\sigma_*^{-1}\Delta)$ is $\sigma$-ample whenever $j\geq 0$.
Hence  $R^1\sigma_*\cO_Y(-(j+1)F)=0$ whenever $j\geq 0$. 
Thus for every $j\geq 0$ we obtain a canonical isomorphism
\[
\fa_j/\fa_{j+1}\cong\sigma_*(\cO_Y(-jF)/\cO_Y(-(j+1)F)).
\]
Since $(Y,F+\sigma_*^{-1}\Delta)$ is plt and $F$ is $\bQ$-Cartier on $Y$, the sheaf $\cO_Y(-jF)$ is Cohen-Macaulay by \cite[Corollary 5.25]{km98}. Hence
the sheaf $\cO_Y(-jF)/\cO_Y(-(j+1)F)$ is also Cohen-Macaulay
whose support is $F$. By \cite[Lemma A.3]{HLS}, we know that over an open set $F^\circ$ of $F$ with $\codim_{F}F\setminus F^\circ\geq 2$, there is a canonical isomorphism between $(\cO_Y(-jF)/\cO_Y(-(j+1)F))|_{F^\circ}$
and $\cO_{F^{\circ}}(\lfloor -jF|_F\rfloor)$. Thus this isomorphism extends to a canonical isomorphism
\[
\cO_Y(-jF)/\cO_Y(-(j+1)F)\cong\cO_F(\lfloor -jF|_F\rfloor).
\]
This implies $\fa_j/\fa_{j+1}\cong H^0(F,\cO_F(\lfloor -jF|_F\rfloor)$ for every $j\geq 0$. Since all the isomorphisms above are canonical, they are all compatible with the obvious product structure of the graded ring. Hence we finish the proof.
\end{proof}

%

The next result is a slight generalization of  \cite[Proposition 5.3]{LX16} (see \cite[Proposition 3.5]{LL19} for an analogous result in terms of conical K\"ahler-Einstein metrics).

\begin{prop}\label{prop:kpscone}
Let $(V,\Delta)$ be an $(n-1)$-dimensional log Fano pair where $\Delta$ has standard coefficients, i.e. its coefficients belong to $\{1-\frac{1}{m}\mid m\in \bZ_{>0}\}$. Let $M$ be an ample $\bQ$-Cartier $\bQ$-divisor on $V$ satisfying $M\sim_{\bQ}-r^{-1}(K_V+\Delta)$ for some $0<r\leq n$ and $\Delta_M=\Delta$.
Denote by $X:=C_p(V,M)$.
Then 
\begin{enumerate}
    \item the pair $(X,(1-\frac{r}{n})V_\infty)$ is  K-semistable if and only if $(V,\Delta)$ is K-semistable.
    \item the pair $(X,(1-\frac{r}{n})V_\infty)$ is  K-polystable if and only if $(V,\Delta)$ is K-polystable.
\end{enumerate}
\end{prop}

\begin{proof}
For the ``if'' part of (1), we follow the proof of \cite[Proposition 5.3]{LX16}. Let $\pi:(\cX,(1-\frac{r}{n})\cV;\cL)\to \bA^1$ be a special test configuration of $(X,(1-\frac{r}{n})V_\infty;L)$ where $L\sim_{\bQ}-l(K_X+(1-\frac{r}{n})V_\infty)$ is an ample line bundle on $X$. By definition, we have 
\[
\Fut(\cX,(1-\tfrac{r}{n})\cV;\cL)=-\frac{(-K_{\ocX/\bP^1}-(1-\tfrac{r}{n})\ocV)^{n+1}}{(n+1)(-K_X-(1-\tfrac{r}{n})V_\infty)^n}.
\]
Since $(1+r)V_\infty\sim_{\bQ}-K_X$, we know that there exists $k\in\bQ$ such that
\[
(1+r)\ocV\sim_{\bQ}-K_{\ocX/\bP^1}+\pi^*\cO_{\bP^1}(k).
\]
Therefore, 
\[
\Fut(\cX,(1-\tfrac{r}{n})\cV;\cL)=-\frac{(\frac{n+1}{n}r\ocV-\pi^*\cO_{\bP^1}(k))^{n+1}}{(n+1)(\frac{n+1}{n}rV_\infty)^n}=k-\frac{r(\ocV^{n+1})}{n(V_\infty^n)}.
\]
Denote by $\cV^\nu$ and $\ocV^\nu$ be the normalization
of $\cV$ and $\ocV$, respectively. Then by adjunction, we have $(K_{\ocX/\bP^1}+\ocV)|_{\ocV^\nu}=K_{\ocV^\nu/\bP^1}+\Delta_{\ocV^\nu}$ where $\Delta_{\ocV^\nu}\geq 0$ is the different. Denote by $\Delta_{\cV^\nu}:=\Delta_{\ocV^\nu}|_{\cV^\nu}$.
Since $(\ocX,\ocV)\times_{\bP^1}(\bP^1\setminus\{0\})\cong (X,V)\times(\bP^1\setminus\{0\})$, we know that 
$\Delta_{\ocV^\nu}|_{\ocV^\nu\setminus\ocV_0^\nu}$ corresponds to $\Delta\times(\bP^1\setminus\{0\})$. Denote by $\Delta_{\tc}$ the Zariski closure of $\Delta\times(\bA^1\setminus\{0\})$ in $\cV^\nu$. Then clearly
$\Delta_{\cV^\nu}\geq \Delta_{\tc}$ and $(\cV^\nu,\Delta_{\tc})/\bA^1$ is a normal test configuration of $(V,\Delta)$. 
It is clear that $L|_{V_\infty}\sim_{\bQ}\frac{(n+1)l}{n} (-K_V-\Delta)$. Denote by $l':=\frac{(n+1)l}{n}$.
Then
\begin{align*}
\Fut(\cV^\nu,\Delta_{\tc};\cL|_{\cV^\nu})&=\frac{1}{(-K_V-\Delta)^{n-1}}\left(\frac{n-1}{n}\cdot\frac{(\ocL|_{\ocV^\nu}^{n})}{l'^{n}}+\frac{(\ocL|_{\ocV^\nu}^{n-1}\cdot (K_{\ocV^\nu/\bP^1}+\overline{\Delta}_{\tc}))}{l'^{n-1}}\right)\\
&\leq\frac{1}{(-K_V-\Delta)^{n-1}}\left(\frac{n-1}{n}\cdot\frac{(\ocL|_{\ocV^\nu}^{n})}{l'^{n}}+\frac{(\ocL|_{\ocV^\nu}^{n-1}\cdot (K_{\ocV^\nu/\bP^1}+\Delta_{\ocV^\nu}))}{l'^{n-1}}\right)\\
& =\frac{1}{(-K_V-\Delta)^{n-1}}\left(\frac{n-1}{n}\cdot\frac{(\ocL^{n}\cdot\ocV)}{l'^{n}}+\frac{(\ocL^{n-1}\cdot\ocV\cdot (K_{\ocX/\bP^1}+\ocV))}{l'^{n-1}}\right)
\end{align*}
Since $-K_{\ocX/\bP^1}\sim_{\bQ}(1+r)\ocV-\pi^*\cO_{\bP^1}(k)$ and $\ocL\sim_{\bQ}l'r\ocV-\pi^*\cO_{\bP^1}(lk)$,
we have
\[
\frac{(\ocL^{n}\cdot\ocV)}{l'^{n}}=r^n(\ocV^{n+1})-\frac{n^2}{n+1}r^{n-1}(\ocV^n\cdot\pi^*\cO_{\bP^1}(k)),
\]
and 
\[
\frac{(\ocL^{n-1}\cdot\ocV\cdot (K_{\ocX/\bP^1}+\ocV))}{l'^{n-1}}
=-r^n(\ocV^{n+1})+\frac{n^2+1}{n+1}r^{n-1}(\ocV^n\cdot\pi^*\cO_{\bP^1}(k)).
\]
Since $(V_\infty^n)=(M^{n-1})=(-K_V-\Delta)^{n-1}/r^{n-1}$
and $(\ocV^n\cdot\cO_{\bP^1}(k))=k(V_\infty^n)$,
we have
\[
\Fut(\cV^\nu,\Delta_{\tc};\cL|_{\cV^\nu})\leq \frac{1}{r^{n-1}(V_\infty^{n})}\left(-\frac{r^n}{n}(\ocV^{n+1})+r^{n-1}k(V_\infty^n)\right)=k-\frac{r(\ocV^{n+1})}{n(V_\infty^n)}.
\]
In particular,  $\Fut(\cV^\nu,\Delta_{\tc};\cL|_{\cV^\nu})\leq \Fut(\cX,(1-\tfrac{r}{n})\cV;\cL)$ with equality if and only if $\Delta_{\cV^\nu}=\Delta_{\tc}$. Thus we finish proving the ``if'' part of (1).
\smallskip

For the ``if'' part of (2), let us take the special test configuration $(\cX,(1-\frac{r}{n})\cV);\cL)/\bA^1$ such that $(\cX_0,(1-\frac{r}{n})\cV_0)$ is K-polystable and $\Fut(\cX,(1-\tfrac{r}{n})\cV;\cL)=0$. Such a special test configuration always exists due to \cite[Theorem 1.3]{LWX18}.
Since $(V,\Delta)$ is K-polystable and $\Fut(\cV^\nu,\Delta_{\tc};\cL|_{\cV^\nu})\leq \Fut(\cX,(1-\tfrac{r}{n})\cV;\cL)$, we know that $\Fut(\cV^\nu,\Delta_{\tc};\cL|_{\cV^\nu})=0$. This implies that $(\cV^\nu,\Delta_{\tc};\cL|_{\cV^\nu})$ is a product test configuration of $(V,\Delta)$ and $\Delta_{\cV^\nu}=\Delta_{\tc}$. By inversion of adjunction, the pair $(\cX,\cV)$ is plt and $\cV^\nu=\cV$. In particular, $(\cX_0,\cV_0)$ is plt with different divisor $\Delta_{\tc,0}$ and   $(\cV_0,\Delta_{\tc,0})\cong (V,\Delta)$. Let us take a $\bQ$-Cartier $\bQ$-divisor $\cM\sim\cV|_{\cV}$ on $\cV$ whose support does not contain $\cV_0$ such that $\cM|_{\cV\setminus\cV_0}\cong M\times(\bA^1\setminus\{0\})$. Thus $\cM\cong M\times\bA^1$ which implies $(\cV_0,\cM_0)\cong (V,M)$ where $\cM_0\sim \cV_0|_{\cV_0}$. Then by Lemma \ref{lem:normalcone} we know that there exists a special test configuration of $(\cX_0,(1-\frac{r}{n})\cV_0)$ with central fiber isomorphic to $(X,(1-\frac{r}{n})V_\infty)$ and vanishing generalized Futaki invariant. Since $(\cX_0,(1-\frac{r}{n})\cV_0)$ is K-polystable, by definition we have $(\cX_0,(1-\frac{r}{n})\cV_0)\cong (X,(1-\frac{r}{n})V_\infty)$ which finishes proving the ``if'' part of (2).

\smallskip

For the ``only if'' parts of (1) and (2), suppose $(\cV,\Delta_{\tc};\cL_{\cV})$ is a special test configuration of $(V,\Delta)$. We will construct a special test configuration $(\cX,(1-\frac{r}{n})\cV;\cL)/\bA^1$ with the same generalized Futaki invariant. The construction goes as follows. Firstly, let us take $\cM$ to be the closure of $M\times(\bA^1\setminus\{0\})$ in $\cV$. Since $\cV_0$ is irreducible and Cartier on $\cV$, we know that  $\cM$ is $\bQ$-Cartier and $\bQ$-linearly equivalent to $-r^{-1}(K_{\cV/\bA^1}+\Delta_{\tc})$. Denote by $\pi:\cV\to \bA^1$. Then let us consider the following variety 
\[
\cX:=\Proj_{\bA^1}\bigoplus_{m=0}^{\infty}\bigoplus_{i=0}^{\infty}\pi_*\cO_{\cV}(\lfloor m\cM\rfloor)\cdot s^i,
\]
where the grading of $\pi_*\cO_{\cV}(\lfloor m\cM\rfloor)$ and $s$ are $m$ and $1$, respectively. We denote the divisor on $\cX$ corresponding to $(s=0)$ by $\cV_\infty$. It is clear that $\cX$ is normal. We will show that $\cV_\infty\cong \cV$ and $\cX_0\cong C_p(\cV_0,\cM_0)$. 

Since $\cV_\infty$ is defined by $(s=0)$, we have
$\cV_\infty=\Proj_{\bA^1}\oplus_{i=0}^{\infty}\pi_*\cO_{\cV}(\lfloor m\cM\rfloor)\cdot s^0
\cong \cV$.
Consider $\{M\}=\sum_{j=1}^k\frac{a_j}{b_j}M_j$ where $M_j$ are distinct prime divisors on $V$ and $0<a_j<b_j$ are coprime integers. Then $\Delta=\Delta_M=\sum_{j=1}^k\frac{b_j-1}{b_j}M_j$. In particular, all non-zero coefficients of $\Delta$ are at least $\frac{1}{2}$. Let $\cM_j$ be the Zariski closure of $M_j$ in $\cV$. Since $(\cV_0,\Delta_{\tc,0})$ is klt, we know that $\cM_{j,0}$ and $\cM_{j',0}$ do not any have component in common for $j\neq j'$. As a result, we have $\lfloor m\cM \rfloor|_{\cV_0}=\lfloor m\cM_0\rfloor$ as $\bZ$-divisors on $\cV_0$. Let us consider the homomorphisms
\[
\pi_*\cO_{\cV}(\lfloor m\cM\rfloor)\otimes\kappa(0)\xrightarrow{g_m} H^0(\cV_0,\cO_{\cV}(\lfloor m\cM\rfloor)\otimes\cO_{\cV_0})\xrightarrow{h_m} H^0(\cV_0,\cO_{\cV_0}(\lfloor m\cM_0\rfloor)).
\]
Denote the composition by $f_m:= h_m\circ g_m$.
Since $\cM$ is $\pi$-ample, by Serre vanishing there exists $m_0\in\bZ_{>0}$ such that $g_m$ is an isomorphism for any $m\geq m_0$. We also know that $\cO_{\cV}(\lfloor m\cM\rfloor)\otimes\cO_{\cV_0}$ and $\cO_{\cV_0}(\lfloor m\cM_0\rfloor)$ agree over the smooth locus of $\cV_0$. Let coherent sheaves $\cG_m$ and $\cG_m'$ on $\cV_0$ be the kernel and cokernel of the morphism $\cO_{\cV}(\lfloor m\cM\rfloor)\otimes\cO_{\cV_0}\to \cO_{\cV_0}(\lfloor m\cM_0\rfloor)$.
Since $\cV_0$ is normal, we know that both supports of $\cG_m$ and $\cG_m'$ have dimension at most $n-3$. Meanwhile, if $m_1\cM$ is Cartier then for any $k\in\bZ_{>0}$ we have
\[
\cG_{m+k m_1}=\cG_m\otimes \cO_{\cV_0}(km_1\cM_0),\quad \cG_{m+k m_1}'=\cG_m'\otimes \cO_{\cV_0}(km_1\cM_0).
\]
Thus we have $\dim\ker(f_m)=O(m^{n-3})$ and $\dim\coker(f_m)=O(m^{n-3})$. It is also clear that $f_m$ is an isomorphism whenver $m\geq m_0$ and $m_1\mid m$. For simplicity we may assume $m_1\geq m_0$. 

Next, we show that $(f_m)$ induces a morphism $\phi:C_p(\cV_0,\cM_0)\to \cX_0$ that is finite, birational, and isomorphic outside a codimension $2$ subset. For simplicity, denote by $P_m:=\pi_*\cO_{\cV}(\lfloor m\cM\rfloor)\otimes\kappa(0)$ and $P_m':=H^0(\cV_0,\cO_{\cV_0}(\lfloor m\cM_0\rfloor))$. Then we know that 
\[
\cX_0=\Proj \bigoplus_{m=0}^{\infty}\bigoplus_{i=0}^{\infty}P_m\cdot s^i
,\qquad C_p(\cV_0,\cM_0) = \Proj \bigoplus_{m=0}^{\infty}\bigoplus_{i=0}^{\infty}P_m'\cdot s^i.
\]
Here the gradings of $P_m$, $P_m'$, and $s$ are $m$, $m$, and $1$, respectively.
From the definition of $f_m:P_m\to P_m'$ we see that it induces ring homomorphism  $\psi: \oplus_{(m,i)
\in \bZ_{\geq 0}^2} P_m\cdot s^i\to \oplus_{(m,i)
\in \bZ_{\geq 0}^2} P_m'\cdot s^i$ which preserves the $\bZ_{\geq 0}^2$-grading. Since $f_m$ is an isomorphism for $m_1\mid m$, we see that $\psi$ induces an isomorphism for the subrings graded by $\{(m,i)\in\bZ_{\geq 0}^2\colon m_1\mid m\}$. This implies that $\psi$ induces a finite morphism $\phi: C_p(\cV_0,\cM_0)\to \cX_0$. Moreover, since both $\ker(f_m)$ and $\coker(f_m)$ have dimension $O(m^{n-3})$, both kernel and cokernel of $\psi_{l}=\oplus_{m=0}^l f_m\cdot s^{l-m}$ have dimension $O(l^{n-2})$. Hence the Hilbert polynomials of both kernel and cokernel of the structure morphism $\phi^{\#}:\cO_{\cX_0}\to \phi_*  \cO_{C_p(\cV_0,\cM_0)}$ have degree at most $n-2$. This implies that $\phi$ is an isomorphism outside a codimension $2$ subset. In particular, $\phi$ is birational. 

From the above properties of $\phi$, we have that $C_p(\cV_0,\cM_0)$ is the normalization of $\cX_0$. Since $\cX_0$ is Cartier in $\cX$ and $C_p(\cV_0,\cM_0)$ is klt, by inversion of adjunction we know that $(\cX,\cX_0)$ is plt. But then $\cX_0$ itself is normal (see e.g. \cite[Proposition 5.51]{km98}) and hence $\phi$ is an isomorphism.
Since $\cV_{\infty,0}$ is the infinity section of $\cX_0$, we know that $\cX_0$ (and hence $\cX$) is smooth at the generic point of $\cV_{\infty,0}$. Therefore the different divisor of $\cV_\infty$ in $\cX$ has no component of $\cV_{\infty,0}$. In particular,  $(K_{\cX/\bA^1}+\cV_\infty)|_{\cV_\infty}=K_{\cV_\infty/\bA^1}+\Delta_{\tc}$. Hence from the discussion of the ``if'' parts, we know $\Fut(\cX,(1-\frac{r}{n})\cV_\infty;\cL)=\Fut(\cV,\Delta_{\tc};\cL_{\cV})$. Hence K-semistability (resp. K-polystability) of $(X,(1-\frac{r}{n})V_\infty)$ implies K-semistability (resp. K-polystability) of $(V,\Delta)$. We finish the proof.
\end{proof}

\begin{lem}\label{lem:normalcone}
Let $X$ be an $n$-dimensional normal projective variety. Let $V$ be a prime divisor on $X$. Assume that  $(X,V)$ is plt and 
 $-K_X\sim_{\bQ}(1+r)V$ is ample for some $0<r\leq n$. 
Let $M\sim V|_V$ be a $\bQ$-Cartier $\bQ$-divisor on $V$.
Then there exists a special test configuration $(\cX,(1-\frac{r}{n})\cV;\cL)$ of $(X,(1-\frac{r}{n})V)$ such that the central fiber $(\cX_0,\cV_0)$ is isomorphic to $(C_p(V,M),V_{\infty})$ and $\Fut(\cX,(1-\frac{r}{n})\cV;\cL)=0$.
\end{lem}

\begin{proof}
The construction is obtain by degeneration to normal cone. We will use the filtration language \cite[Section 2.3]{BX18} to construct $\cX$. Let $R:=\oplus_{m=0}^\infty H^0(X,\cO_X(mV))$ be the graded ring of $X$. Let $R_m$ be the $m$-th graded piece of $R$. Consider the $\bN$-filtration $\cF$ of $R$ defined as
\[
\cF^\lambda R_m=H^0(X,\cO_X((m-\lambda )V))\subset H^0(X,\cO_X(mV))=R_m\quad\textrm{ for any }\lambda\in\bZ_{\geq 0}.
\]
Note that $\cF^\lambda R_m=R_m$ for any non-positive integer $\lambda$.
The Rees algebra of $\cF$ is defined as
\[
\Rees(\cF):=\bigoplus_{m=0}^\infty\bigoplus_{\lambda=-\infty}^\infty t^{-\lambda}\cF^\lambda R_m\subset R[t,t^{-1}].
\]
The associated graded ring of $\cF$ is defined as
\[
\gr_{\cF} R:=\bigoplus_{m=0}^\infty\bigoplus_{\lambda=-\infty}^\infty \gr_{\cF}^\lambda R_m,
\quad \textrm{ where }\gr_{\cF}^\lambda R_m:=\frac{\cF^\lambda R_m}{\cF^{\lambda+1}R_m}.
\]
Note that the gradings of $\Rees(\cF)$ and $\gr_{\cF} R$ are both given by $m$. In our set-up it is easy to see that both $\Rees(\cF)$ and $\gr_{\cF}R$ are finitely generated.
Let $\cX:=\Proj_{\bA^1}\Rees(\cF)$. Then from \cite[Section 2.3.1]{BX18} we know that 
$\cX_0\cong \Proj~\gr_{\cF} R$. We will show that $\cX_0\cong C_p(V,M)$.

When $\lambda\leq -1$, we know $\cF^\lambda R_m=\cF^{\lambda+1} R_m= R_m$ so $\gr_{\cF}^\lambda R_m=0$. 
When $\lambda\geq m+1$, we know $\cF^{\lambda}R_m= \cF^{\lambda+1} R_m=0$ so $\gr_{\cF}^\lambda R_m=0$. Let us assume $\lambda\in [0,m]\cap\bZ$.
By Kawamata-Viehweg vanishing, we know $H^1(X,\cO_{X}((m-\lambda-1)V))=0$  since $(m-\lambda-1)V-K_X=(m-\lambda+r)V$ is ample. Hence from the proof of Proposition \ref{prop:kcdeg} we have a canonical isomorphism
\[
\gr_{\cF}^\lambda R_m=H^0(X,\cO_{X}((m-\lambda)V))/H^0(X,\cO_{X}((m-\lambda-1)V))\cong H^0(V,\cO_V(\lfloor(m-\lambda)M\rfloor)).
\]
Thus $\gr_{\cF} R\cong \oplus_{m=0}^\infty\oplus_{i=0}^\infty H^0(V,\cO_V(\lfloor mM\rfloor))\cdot s^i$ as graded rings, so $\cX_0\cong C_p(V,M)$.

So far we have constructed a variety $\cX$ which provides a special degeneration from $X$ to $\cX_0\cong C_p(V,M)$. Let $\cV$ be the Zariski closure of $V\times(\bA^1\setminus\{0\})$ in $\cX$. Denote the graded ideal of $V$ in $X$ by $I_V\subset R$. Then it is clear that $I_{V,m}=H^0(X,\cO_X((m-1)V))$. By \cite[Section 2.3.1]{BX18}, the graded ideal of the scheme-theoretic central fiber $\cV_0$ is given by the initial ideal $\initl(I_V)$ where
\[
\initl(I_V):=\bigoplus_{m=0}^\infty\bigoplus_{\lambda=-\infty}^\infty \mathrm{im}(\cF^\lambda R_m\cap I_V\to \gr_{\cF}^\lambda R_m) \subset \gr_{\cF} R.
\]
We may restrict to $\lambda\in [0,m]\cap\bZ$. Then it is easy to see that 
\[
\mathrm{im}(\cF^\lambda R_m\cap I_V\to \gr_{\cF}^\lambda R_m)=\begin{cases}
\gr_{\cF}^\lambda R_m & \textrm{ if }\lambda\geq 1,\\
0 & \textrm{ if }\lambda\leq 0.
\end{cases}
\]
Thus $\initl(I_V)=\oplus_{m=0}^\infty\oplus_{\lambda=1}^m \gr_{\cF}^\lambda R_m$. This implies that $\cV_0\cong V_\infty$ under the isomorphism $\cX_0\cong C_p(V,M)$. 
Therefore, $(\cX,(1-\frac{r}{n})\cV;\cL)$ is a special test configuration of $(X,V;L)$ where $\cL:=-l(K_{\cX/\bA^1}+(1-\frac{r}{n})\cV)$ is Cartier for some $l\in\bZ_{>0}$. Since the $\lambda$-grading on $\gr_{\cF} R/\initl(I_V)$ is trivial, we know that $(\cV,\Delta_{\tc};\cL|_{\cV})$ is a trivial test configuration of $(V,\Delta;L|_V)$ where $\Delta_{\tc}$ is the Zariski closure of $\Delta\times(\bA^1\setminus\{0\})$ in $\cV$.
Since $\cX_0$ is normal and Cartier in $\cX$, we know that $\cX$ is smooth at the generic point of $\cV_0$. Then by normality of $\cV$, we know $(K_{\cX/\bA^1}+\cV)|_{\cV}=K_{\cV/\bA^1}+\Delta_{\tc}$. Hence by exactly the same computation in the proof of ``if'' part of Proposition \ref{prop:kpscone}(1), we know
\[
\Fut(\cX,(1-\tfrac{r}{n})\cV;\cL)=\Fut(\cV,\Delta_{\tc};\cL|_{\cV})=0.
\]
The proof is finished.
\end{proof}

\section{Fujita's characterization}\label{sec:fuj}

In this section, we make a few observations on the restrictions on $\bQ$-Fano varieties $X$ with $\alpha(X)=\frac{\dim(X)}{\dim(X)+1}$ that are not K-stable. We first have the following characterization, essentially due to K. Fujita.

\begin{thm}[cf. {\cite{Fuj16}}] \label{thm:fujalpha}
Let $X$ be an $n$-dimensional $\bQ$-Fano variety.
Assume that $\alpha(X)= \frac{n}{n+1}$ and $X$ is not
K-stable. Then there exists a birational morphism $\sigma:(F\subset Y)\to (x\in X)$ extracting a Koll\'ar component $F$ over $x\in X$, such that the following properties hold:
\begin{enumerate}
    \item $A_X(F)=n$, $\tau(F)=\epsilon(F)=n+1$.
    \item The divisor $\sigma^*(-K_X)-(n+1)F$ is semiample
    but not big. The ample model $\pi: Y\to Z$ of $\sigma^*(-K_X)-(n+1)F$ satisfies that all fibers of $\pi$ are one-dimensional and a generic fiber is $\bP^1$. Moreover, $\pi|_{F}:F\to Z$ is an isomorphism.
    \item $\alpha(F,\Delta_F)\geq 1$ where $\Delta_F$ is the different. Moreover, $\ord_F$ is a minimizer of the normalized volume function $\hvol$ on $\Val_{X,x}$.
\end{enumerate}
In particular, $x\in X$ is a weakly exceptional singularity.
\end{thm}

\begin{proof}
Since $X$ is not K-stable, by Theorem \ref{thm:beta-criterion} there exists a dreamy prime divisor $F$ over $X$ such that $\beta(F)\leq 0$. Meanwhile,
\cite[Lemma 3.3]{FO16} implies that $A_X(F)\geq \alpha(X)\tau(F)=\frac{n}{n+1}\tau(F)$. Then by \cite[Lemma 3.1, Proposition 3.2 and 3.3]{Fuj16} we know that there exists a proper birational morphism $\sigma:Y\to X$ with $Y$ normal such that $F=\Exc(\sigma)$ is a prime divisor on $Y$, $-F$ is $\sigma$-ample, $\sigma(F)$ is a closed point $x\in X$, and parts (1) and (2) hold. Moreover, part (3) on $\alpha(F,\Delta_F)\geq 1$ implies that $(F, \Delta_F)$ is klt, i.e. $F$ is a Koll\'ar component. Thus it suffices to show that part (3) holds.

Finally we prove part (3). From part (2) we know that $F\cong Z$ are both normal.  Denote by $\Delta_{F}$ the different divisor of $F$ in $Y$. Let $B\sim_{\bQ} -K_{F}-\Delta_{F}$ be any effective $\bQ$-divisor on $F$. We denote by $B_Z:=(\pi|_{F})_*B$. From parts (1) and (2) we know that  
\[
(-K_Y-2F)|_{F}=(\sigma^*(-K_X)-(n+1)F)|_{F}=-(n+1)F|_{F}.
\]
Hence $-K_{F}-\Delta_{F}=(-K_Y-F)|_{F}=-n F|_{F}$. Since $\pi|_F:F\to Z$ is an isomorphism, we know that 
\[
\sigma^*(-K_X)-(n+1)F\sim_{\bQ}\frac{n+1}{n}\pi^*(\pi|_{F})_*(-K_{F}-\Delta_{F})\sim_{\bQ}\frac{n+1}{n}\pi^*B_Z.
\]
It is clear that $\pi^*B_Z$ does not contain $F$ as a component. Thus $B_X:=\sigma_*\pi^*B_Z$ satisfies that 
$B_X\sim_{\bQ}-\frac{n}{n+1}K_X$ and $\ord_F(B_X)=n$.
Since $\alpha(X)=\frac{n}{n+1}$, the pair $(X,B_X)$ is log canonical. We know
\[
\sigma^*(K_X+B_X)=K_Y-(n-1)F+\pi^*B_Z+nF=K_Y+F+\pi^*B_Z.
\]
Hence $(Y,F+\pi^*B_Z)$ is log canonical. By adjunction, 
$(F,\Delta_{F}+B)$ is log canonical since $B=(\pi^*B_Z)|_{F}$. By choosing $B$ whose support contains a log canonical center of $(F, \Delta_{F})$, we see that $(F, \Delta_{F})$ is klt. Hence inversion of adjunction implies that $F$ is a Koll\'ar component over $x\in X$ and $\alpha(F,\Delta_F)\geq 1$. 

By part (2) we know that 
\[
(\sigma^*(-K_X)-(n+1)F)^n=(-K_X)^n-(n+1)^n (-1)^{n-1}(F^n)=0.
\]
It is clear that $\vol_{X,x}(\ord_F)=(-1)^{n-1}(F^n)$. From part (1) we know that $A_X(F)=n$. Thus we get $(-K_X)^n=(1+\frac{1}{n})^n\hvol_{X,x}(\ord_F)$. Since $\alpha(X)=\frac{n}{n+1}$, Theorem \ref{thm:tian} implies that $X$ is K-semistable. Hence \cite[Theorem 21]{Liu18} implies that $\ord_F$ is a divisorial minimizer of the normalized volume function $\hvol$ on $\Val_{X,x}$. Thus the proof of part (3) is finished.
\end{proof}

A quick consequence of the proof of Theorem \ref{thm:fujalpha} is that every $\bQ$-Fano variety $X$ satisfying $\alpha(X)=\frac{n}{n+1}$ that is not K-stable always specially degenerates to a K-polystable $\bQ$-Fano variety $\cX_0$ with $\alpha(\cX_0)=\frac{1}{n+1}$.

\begin{cor} \label{cor:polystable degeneration}
Let $X$ be an $n$-dimensional $\bQ$-Fano variety.
Assume that $\alpha(X)= \frac{n}{n+1}$ and $X$ is not
K-stable. Then there exists a special test configuration $(\cX,\cL)$ of $X$ such that $\mathrm{Fut}(\cX;\cL)=0$. Moreover, $\cX_0$ is K-polystable and $\alpha(\cX_0)=\frac{1}{n+1}$. In particular, if $n\geq 2$ then $X$ is not K-polystable.
\end{cor}

\begin{proof}
Denote by $\fa_j:=\{f\in\cO_{X,x}\mid \ord_F(f)\geq j\}$ for $j\in\bZ$.
From the proof of \cite[Lemma 33]{Liu18}, we know that there exists a special test configuration $(\cX;\cL)$ such that $\Fut(\cX;\cL)=0$ and $\cX_0\cong\Proj (\oplus_{j=0}^\infty\fa_j/\fa_{j+1})[s]$. Since $F$ is a Koll\'ar component over $x\in X$, by adjunction of plt pairs we know that $F|_F$ is a $\bQ$-divisor on $F$ that is well-defined up to $\bZ$-linear equivalence. Moreover, if $\{F|_F\}=\sum_i\frac{a_i}{b_i}B_i$ for $a_i< b_i$ coprime positive integers and $B_i$ prime divisors on $F$, then (using the same notation as in Theorem \ref{thm:fujalpha}) $\Delta_F=\frac{b_i-1}{b_i}B_i$ (see e.g. \cite[Section 16]{FA92}). By \cite[Section 2.4]{LX16} or Proposition \ref{prop:kcdeg}, we know that
$\oplus_{j=0}^\infty\fa_j/\fa_{j+1}\cong\oplus_{m=0}^\infty H^0(F,\cO_F(\lfloor -mF|_F\rfloor))$ as graded rings which implies $\cX_0\cong C_p(F,(-F)|_F)$. Since $\alpha(F,\Delta_F)\geq 1$, Theorem \ref{thm:tian} implies that $(F,\Delta_F)$ is K-stable. Hence Proposition \ref{prop:kpscone} implies that $\cX_0$ is K-polystable since   $-K_F-\Delta_F\sim_{\bQ}-n F|_F$. To show $\alpha(\cX_0)=\frac{1}{n+1}$, we notice that $\alpha(\cX_0)\geq \frac{1}{n+1}$ by \cite[Theorem 3.5]{FO16}, and $\alpha(\cX_0)\leq\frac{1}{n+1}$ since $-K_{\cX_0}\sim_{\bQ} (n+1)F_{\infty}$. Hence the proof is finished.
\end{proof}

The equality $\alpha(X)+\alpha(\cX_0)=1$ here is not a coincidence, as can be seen by the following statement:

\begin{prop} \label{prop:alpha sum>1}
Let $f:X\rightarrow C$, $g:Y\rightarrow C$ be two $\bQ$-Gorenstein flat families of $\bQ$-Fano varieties $($i.e. all geometric fibers are integral, normal and $\bQ$-Fano, $K_{X/C}$ and $K_{Y/C}$ are both $\bQ$-Cartier$)$ over a smooth pointed curve $0\in C$. Let $X_0=f^{-1}(0)$ and $Y_0=g^{-1}(0)$. Assume that there exists an isomorphism $\rho:X\backslash X_0 \xrightarrow{\cong} Y\backslash Y_0$ over the punctured curve $C\backslash 0$ and $\alpha(X_0)+\alpha(Y_0)>1$. Then $\rho$ induces an isomorphism $X\cong Y$ over $C$.
\end{prop}

\begin{proof}
By assumption $X$ is birational to $Y$ over $C$. Let $m$ be a sufficiently large and divisible integer and let $D_1\in |-mK_{X_0}|$, $D_2\in |-mK_{Y_0}|$ be general divisors in the corresponding linear systems. Choose effective divisors $D_{X,1}\sim_\bQ -mK_X$, $D_{Y,2}\sim_\bQ -mK_Y$ not containing $X_0$ or $Y_0$ such that $D_{X,1}|_{X_0}=D_1$ and $D_{Y,2}|_{Y_0}=D_2$. Let $D_{Y,1}$ and $D_{X,2}$ be their strict transforms to the other family. Let 
$D_X=\frac{\alpha(Y_0)}{\alpha(X_0)+\alpha(Y_0)}D_{X,1}+\frac{\alpha(X_0)}{\alpha(X_0)+\alpha(Y_0)}D_{X,2}$ and similar $D_Y=\frac{\alpha(Y_0)}{\alpha(X_0)+\alpha(Y_0)}D_{Y,1}+\frac{\alpha(X_0)}{\alpha(X_0)+\alpha(Y_0)}D_{Y,2}$. Apparently $D_Y$ is the strict transform of $D_X$. As $D_1$ is a general member of a very ample linear system and $D_{X,2}|_{X_0}\sim_\bQ -mK_{X_0}$, by the definition of alpha invariant we see that $(X_0,\frac{1}{m}D_X|_{X_0})$ is klt. Similarly $(Y_0,\frac{1}{m}D_Y|_{Y_0})$ is klt. The result then follows from the next lemma.
\end{proof}

\begin{lem}
Let $f:X\rightarrow C$, $g:Y\rightarrow C$ be $\bQ$-Gorenstein flat families of $\bQ$-Fano varieties over a smooth pointed curve $0\in C$ with central fibers $X_0$ and $Y_0$. Let $D_X\sim_\bQ -K_X$, $D_Y\sim_\bQ -K_Y$ be effective $\bQ$-divisors not containing $X_0$ or $Y_0$. Assume that there exists an isomorphism 
\[\rho:(X,D_X)\times_C C^\circ \cong (Y,D_Y)\times_C C^\circ\]
over $C^\circ=C\backslash 0$, that $(X_0,D_X|_{X_0})$ is klt and $(Y_0,D_Y|_{Y_0})$ is lc. Then $\rho$ extends to an isomorphism $(X,D_X)\cong (Y,D_Y)$.
\end{lem}

\begin{proof}
This follows from the exact same proof of \cite[Theorem 5.2]{LWX} (see also \cite[Proposition 3.2]{BX18}).
\end{proof}

\begin{rem}
When the total spaces have terminal singularities, Proposition \ref{prop:alpha sum>1} is proved by Cheltsov \cite[Theorem 1.5]{C-alpha>1}.
\end{rem}

\begin{proof}[Proof of Theorem \ref{thm:alpha-moduli}]
Boundedness  follows from \cite{Jia-boundedness};  openness follows from \cite[Theorem B]{BL18b}; separatedness follows from Proposition \ref{prop:alpha sum>1}; the existence of coarse moduli space is due to Keel and Mori \cite{KM97}. 
\end{proof}

We believe such moduli spaces are interesting and deserve further investigation. See \cite[Corollary 1.4]{BX18} for an analogous result on moduli of uniformly K-stable $\bQ$-Fano varieties.

We also give an alternative proof of the finiteness of automorphism group in Theorem \ref{thm:alpha-moduli} without using stacks.

\begin{cor}
Let $X$ be a $\bQ$-Fano variety such that $\alpha(X)>\frac{1}{2}$, then $\Aut(X)$ is finite.
\end{cor}

\begin{proof}
Assume that $\Aut(X)$ is not finite, then it contains $\bG_m$ or $\bG_a$. Let $\cX_1=\cX_2=X\times \bA^1$. If $\bG_m\subseteq\Aut(X)$, then it induces an isomorphism $\cX_1\times_{\bA^1} (\bA^{1}\setminus\{0\}) \stackrel{\sim}{\longrightarrow} \cX_2\times_{\bA^1} (\bA^{1}\setminus\{0\})=X\times (\bA^{1}\setminus\{0\})$ through the diagonal action of $\bG_m$. Since $\alpha(X)>\frac{1}{2}$, this map extends to an isomorphism of $\cX_1$ and $\cX_2$ over $\bA^1$ by Proposition \ref{prop:alpha sum>1}. Thus the map $\bG_m\rightarrow\Aut(X)$ extends to a map $\bA^1\rightarrow\Aut(X)$, a contradiction. Similarly, if $\bG_a\subseteq\Aut(X)$, then the inclusion $\bA^{1}\setminus\{0\}\subseteq \bG_a$ given by $t\mapsto t^{-1}$ induces an automorphism of $X\times(\bA^{1}\setminus\{0\})$ over $\bA^{1}\setminus\{0\}$ which extends to an automorphism over $\bA^1$, hence $\bG_a\rightarrow\Aut(X)$ extends to a map $\bP^1\rightarrow \Aut(X)$, again a contradiction.
\end{proof}

Another quick consequence of Theorem \ref{thm:fujalpha} is a refinement of Tian's criterion in small dimension. Note that the dimension $2$ case was proved by K. Fujita \cite{Fujita-private}.

\begin{cor}
Let $X$ be a $\bQ$-Fano variety such that $\alpha(X)= \frac{n}{n+1}$ where $n=\dim X$. Assume that either $n=2$ or $n=3$ and $X$ has terminal singularities. Then $X$ is K-stable.
\end{cor}

\begin{proof}
By Theorem \ref{thm:fujalpha}, it suffices to show that under the given assumptions, $X$ cannot have a weakly exceptional singularity whose corresponding Koll\'ar component has log discrepancy $n$. For 3-fold terminal singularities, this follows from \cite[Corollary 4.8]{Pro00}. In the surface case, it is well known (and not hard to verify) that a klt surface singularity is weakly exceptional if and only if the dual graph of the exceptional divisors on the minimal resolution has a fork, in which case the Koll\'ar component lives on the minimal resolution and corresponds to the central fork. In particular, the log discrepancy of the Koll\'ar component is at most $1$ and does not satisfy our requirement.
\end{proof}

From Corollary \ref{cor:polystable degeneration} we see that K-semistable $\bQ$-Fano varieties with smallest $\alpha$-invariant naturally occur in studying non-K-stable $\bQ$-Fano varieties with largest $\alpha$-invariant. In \cite[Theorem 1.2]{Jia17}, Jiang showed that a K-semistable Fano manifold $X$ with $\alpha(X)=\frac{1}{\dim(X)+1}$ is isomorphic to projective space. The following result provides a characterization of possibly singular such $X$.

\begin{thm}\label{thm:jiang}
Let $X$ be an $n$-dimensional K-semistable $\bQ$-Fano variety satisfying that $\alpha(X)=\frac{1}{n+1}$. Then there exists an $(n-1)$-dimensional K-polystable log Fano pair $(V,\Delta)$ where $\Delta$ has standard coefficients, and $M$ an ample $\bQ$-Cartier $\bQ$-divisor on $V$ satisfying $nM\sim_{\bQ}-K_V-\Delta$ and $\Delta_M=\Delta$, such that $C_p(V,M)$ is the unique K-polystable special degeneration of $X$. In addition, $\alpha(C_p(V,M))=\frac{1}{n+1}$.
\end{thm}

\begin{proof}
Let $X'$ be the unique K-polystable special degeneration of $X$ whose existence is proved in \cite[Theorem 1.3]{LWX18}. Then we have $\alpha(X')\leq \alpha(X)=\frac{1}{n+1}$ by \cite[Theorem B]{BL18b} and $\alpha(X')\geq \frac{1}{n+1}$ by \cite[Theorem 3.5]{FO16}. Thus $\alpha(X')=\frac{1}{n+1}$. Then by \cite[Proposition 3.1]{Jia17}, there exists a prime divisor $V$ on $X'$ such that $(X',V)$ is plt and $-K_{X'}\sim_{\bQ}(n+1)V$. Let $M\sim V|_V$ be a $\bQ$-Cartier $\bQ$-divisor on $V$. Thus Lemma \ref{lem:normalcone} implies that there exists a special test configuration $(\cX';\cL')$ of $X'$ such that the central fiber $\cX_0'$ is isomorphic to $C_p(V,M)$ and $\Fut(\cX';\cL')=0$. Since $X'$ is K-polystable, we have $X'\cong\cX_0'\cong C_p(V,M)$. Denote $\Delta:=\Diff_V(0)$, then it is clear that $\Delta$ has standard coefficients, $\Delta=\Delta_M$, and 
\[
-K_V-\Delta=(-K_X-V)|_V\sim_{\bQ} nV|_V=nM.
\]
Since $C_p(V,M)\cong X'$ is K-polystable, the log Fano pair $(V,\Delta)$ is also K-polystable by Proposition \ref{prop:kpscone}. Hence we finish the proof.
\end{proof}

The next result shows that in Theorem \ref{thm:fujalpha} the different divisor $\Delta_F$ never vanishes. Note that Proposition \ref{prop:Delta-nonzero} and Lemma \ref{lem:diffzero} are not needed in the rest of the paper.

\begin{prop}\label{prop:Delta-nonzero}
Assume $n\geq 2$. Then with the notation of Theorem \ref{thm:fujalpha}, we have $\Delta_F\neq 0$.
\end{prop}

\begin{proof}
Assume to the contrary that $\Delta_F=0$. Then by Theorem \ref{thm:fujalpha} and Lemma \ref{lem:diffzero}, we know that $X\cong C_p(F,\cO_F(-F))$ and $\alpha(X)\leq \frac{1}{n+1}$ which contradicts to $\alpha(X)=\frac{n}{n+1}$.
\end{proof}

\begin{lem}\label{lem:diffzero}
 Let $X$ be a $\bQ$-Fano variety of dimension $n\geq 2$. Let $\sigma: Y\to X$ be
 a proper birational morphism that provides a Koll\'ar component $F$ with 
 $x=\sigma(F)$. Assume that $\epsilon(F)=\tau(F)>A_X(F)$. If $\cO_Y(F)$
 is locally free in codimension $2$, then $X\cong C_p(F,\cO_F(-F))$.
 In particular, $\epsilon(F)=A_X(F)+1$ and $\alpha(X)\leq \frac{1}{A_X(F)+1}$.
\end{lem}

\begin{proof}
 Let  $D:=\sigma^*(-K_X)-\epsilon(F)\cdot F$. By definition we know that
 $D$ is a nef $\bR$-divisor that is not big. Since $\epsilon(F)>A_X(F)>0$, 
 $-K_Y-F=\sigma^*(-K_X)-A_X(F)\cdot F$ is ample. Hence the cone theorem implies
 that $\overline{NE}(Y)$ is polyhedral, so $\epsilon(F)\in\bQ$. By Shokurov's 
 basepoint-free theorem and Kawamata-Viehweg vanishing theorem, the $\bQ$-divisor
 $D$ is semiample and defines an algebraic fiber space $\pi: Y\to Z$. Since
 $-K_Y-F=D+(\epsilon(F)-A_X(F))F$ is ample, $F$ is a $\pi$-ample $\bQ$-Cartier divisor.
 Following the same argument as in \cite[Proposition 3.3]{Fuj16}, the fibers of
 $\pi$ are one-dimensional, $\pi$ is a generic
 $\bP^1$-fibration, and $\pi|_F:F\to Z$ is an isomorphism. In particular, $-K_Y-2F$ has zero intersection number with
 a generic fiber of $\pi$, hence $-K_Y-2F=D$ which yields $\epsilon(F)=A_X(F)+1$.
 
 Let $F^0$ be the biggest open subset of the smooth locus $F_{\reg}$ of $F$ such that
 $\cO_Y(F)$ is locally free along $F^0$. By assumption we have $\codim_F 
 F\setminus F^0\geq 2$.
 Denote by $Y^0:=Y\setminus(F\setminus F^0)$ with $i: Y^0\hookrightarrow Y$ and $j:F^0\hookrightarrow F$ the open immersions,
 hence $\codim_Y Y\setminus Y^0\geq 3$.
 Denote by $Z^0:=\pi(F^0)$.
 Since $\cO_Y(F)$ is locally free along $F^0$, we have the short exact sequence
 \begin{equation}\label{exseq1}
  0\to \cO_{Y^0}((m-1)F^0)\to\cO_{Y^0}(mF^0)\to \cO_{F^0}(mF^0|_{F^0})\to 0.
 \end{equation}
 By taking $i_*$, we get an exact sequence
 \begin{equation}\label{exseq2}
  0\to i_*\cO_{Y^0}((m-1)F^0)\to i_*\cO_{Y^0}(mF^0)\to j_*\cO_{F^0}(mF^0|_{F^0})
  \to R^1 i_*\cO_{Y^0}((m-1)F^0).
 \end{equation}
 Since $Y^0$ and $F^0$ are open subsets of $Y$ and $F$ respectively whose complements have codimension at least $2$,
 we have these natural isomorphisms
 \[
  i_*\cO_{Y^0}(mF^0)\cong \cO_Y(mF),\quad j_*\cO_{F^0}(mF^0|_{F^0})\cong \cO_F(mF|_F).
 \]
 From the assumption that $F$ is Cartier in codimension $2$, we have that
 $\cO_F(mF|_F)$ is a well-defined $\bQ$-Cartier Weil divisorial sheaf on $F$
 satisfying $\cO_F(mF|_F)\cong\cO_F(F|_F)^{[m]}$. 
 By \cite[Corollary 5.25]{km98} we have that $\cO_Y(mF)$ is Cohen-Macaulay for any 
 $m\in \bZ$. Since $\codim_Y Y\setminus Y^0\geq 3$, the local cohomology 
 long exact sequence implies that 
 \[
  R^1 i_* \cO_{Y^0}((m-1)F^0)\cong \cH^2_{Y\setminus Y^0}(Y,\cO_Y((m-1)F))=0.
 \]
 Hence the exact sequence \eqref{exseq2} becomes
 \begin{equation}\label{exseq3}
    0\to \cO_Y((m-1)F)\to\cO_Y(mF)\to \cO_F(mF|_F)\to 0.
 \end{equation}
 Let $M$ be a $\bQ$-Cartier Weil divisor on $F$ representing $\cO_F(-F)$,
 then $M|_{F^0}$ is Cartier on $F^0$. Denote by $M_Z:=(\pi|_F)_*M$. After 
 tensoring \eqref{exseq3} by $\cO_Y(m\cdot\pi^*M_Z)$ and taking the reflexive hull,
 we get an exact sequence 
 \begin{equation}\label{exseq4}
  0\to \cO_Y((m-1)F+m\cdot\pi^*M_Z))\to\cO_Y(m(F+\pi^*M_Z))\to \cO_F\to 0.
 \end{equation}
 The reason that \eqref{exseq4} is exact can be deduced as follows: we first tensor
 \eqref{exseq1} by $\cO_{Y^0}(m\cdot\pi^*M_Z)$ (since $\cO_{Y^0}(m\cdot\pi^*M_Z)$
 is locally free along $F^0$, the exactness remains), then taking $i_*$ (since 
 $\pi^*M_Z$ is a $\bQ$-Cartier Weil divisor on $Y$,
 we again use \cite[Corollary 5.25]{km98} to show that $\cO_Y((m-1)F+m\cdot\pi^*M_Z))$ is 
 Cohen-Macaulay); then the argument
 above works in a similar way.
 
 Now applying $\pi_*$ to the exact sequence \eqref{exseq4} when $m=1$, we get
 an exact sequence
 \begin{equation}\label{exseq5}
  0\to \pi_*\cO_Y(\pi^*M_Z)\to\pi_*\cO_Y(F+\pi^*M_Z)\to \cO_Z\to R^1\pi_*\cO_Y(\pi^*M_Z).
 \end{equation}
 It is clear that $K_Y\sim_{\bQ,\pi}-2F$ is $\pi$-anti-ample, hence Kawamata-Viehweg
 vanishing implies that $R^1\pi_*\cO_Y(\pi^*M_Z)=0$. The projection formula yields
 $(\pi_*\cO_{Y}(\pi^*M_Z))|_{Z^0}\cong \cO_{Z^0}(M_Z)$.
 As abuse of notation we also denote the open immersion $Z^0\hookrightarrow Z$ by $j$. Hence 
 \[
  \pi_*\cO_Y(\pi^*M_Z)\cong j_*\left((\pi_*\cO_{Y}(\pi^*M_Z))|_{Z^0}\right)
  \cong j_* \cO_{Z^0}(M_Z)\cong \cO_{Z}(M_Z),
 \]
 where the first equality follows from the fact that $\cO_Y(\pi^*M_Z)$ satisfies $S_2$-condition and $\pi$ is flat.
 As a result, \eqref{exseq5} becomes
 \begin{equation}\label{exseq6}
  0\to \cO_Z(M_Z)\to\pi_*\cO_Y(F+\pi^*M_Z)\to \cO_Z\to 0.
 \end{equation}
 Hence $\pi_*\cO_Y(F+\pi^*M_Z)$ is an extension of $\cO_Z$ by $\cO_Z(M_Z)$. 
 It is clear that
 \[
 \Ext^1(\cO_Z,\cO_Z(M_Z))\cong H^1(Z, \cO_Z(M_Z))\cong H^1(F,\cO_F(M)).
 \]
 By assumption we have that 
 \[
  K_F=(K_Y+F)|_F\sim_{\bQ}(\sigma^*K_X+A_X(F)\cdot F)|_F =-A_X(F)\cdot M.
 \]
 Hence $M-K_F\sim_{\bQ}(A_X(F)+1)M$ is ample.
 Then Kawamata-Viehweg vanishing yields $H^1(F,\cO_F(M))=0$. So the exact sequence
 \eqref{exseq6} splits, i.e. $\pi_*\cO_Y(F+\pi^*M_Z)\cong \cO_Z(M_Z)\oplus\cO_Z$.
 By restricting to $Z^0$, tensoring with $\cO_{Z^0}(-M_Z)$ and taking $i_*$,
 we have 
 \[
  \pi_*\cO_Y(F)\cong\cO_Z\oplus\cO_Z(-M_Z).
 \]
 
 Denote by $Y':=\pi^{-1}(Z^0)$.
 From \cite[Proposition 3.3]{Fuj16} and \cite[Lemma 5]{LZ18} we know that $\pi^0=\pi|_{Y'}:Y'\to Z^0$ is a smooth
 $\bP^1$-fibration. Since $\pi^0$ admits a section, it is a $\bP^1$-bundle.
 As a consequence, 
 \begin{equation}\label{eq1}
  \bigoplus_{m=0}^\infty(\pi^0)_*\cO_{Y'}(mF^0)\cong \bigoplus_{m=0}^\infty\Sym^m \big((\pi^0)_*\cO_{Y'}
 (F^0)\big).
 \end{equation}
 Applying $j_*$ to \eqref{eq1} yields
 \begin{align*}
  \bigoplus_{m=0}^\infty\pi_*\cO_{Y}(mF)&\cong \bigoplus_{m=0}^\infty j_*~\Sym^m \big((\pi^0)_*\cO_{Y'}
 (F^0)\big)\\
 &\cong \bigoplus_{m=0}^\infty j_*\left(\bigoplus_{k=0}^m\cO_{Z^0}(-kM_Z)\right)\\
 &\cong \bigoplus_{m=0}^\infty \bigoplus_{k=0}^m\cO_{Z}(-kM_Z).
 \end{align*}
 Since $F$ is $\pi$-ample, we have
 \[
  Y\cong \Proj_{Z}\bigoplus_{m=0}^\infty\pi_*\cO_{Y}(mF)\cong \Proj_Z\bigoplus_{m=0}^\infty\bigoplus_{k=0}^m\cO_{Z}(-kM_Z).
 \]
 Hence $(Y,F)$ is isomorphic to the canonical blow-up of the projective orbifold cone
 $C_p(F,\cO_F(M))$ at the vertex, which implies $X\cong C_p(F,\cO_F(M))$.
 
 Since $-K_F\sim_{\bQ} A_X(F)\cdot M$, we have that $-K_X\sim_{\bQ} (A_X(F)+1) F_{\infty}$
 where $F_\infty$ is the section of the projective orbifold cone at infinity.
 Hence $\alpha(X)\leq \frac{1}{A_X(F)+1}$.
\end{proof}

\section{Log canonical thresholds on general complete intersections} \label{sec:lct-hyp}

Let $r\ge 1$, $n\ge 3$ and $d_1, \cdots, d_r\ge 2$ be integers. In this section we study the (global) log canonical threshold of the hyperplane class $H$ on a general smooth complete intersection $X=F_1\cap \cdots \cap F_r \subseteq \bP^{n+r}$ of codimension $r$ and dimension $n$, where each $F_i$ is a hypersurface of degree $d_i$. In particular, we construct (Corollary \ref{cor:weakly exceptional}) weakly exceptional singularities with the properties described in Theorem \ref{thm:fujalpha} and also prove Theorem \ref{thm:lct-general-hyp}. First let us specify the generality condition we will consider. For this we recall the regularity condition introduced by Pukhlikov (see e.g. \cite{Puk-cpi} or \cite[\S 3.2]{Puk-book}). Let $x\in X$ and let $f_i$ be the defining equation of $F_i$. Choose a system of affine linear coordinates $z_*=(z_1,\cdots,z_{n+r})$ with origin at $x$ and write $f_i$ as
\[f_i=q_{i,1}+q_{i,2}+\cdots+q_{i,d_i}
\]
where $q_{i,j}=q_{i,j}(z_*)$ is homogeneous of degree $j$. We rearrange the $d=\sum _{i=1}^r d_i$ polynomials $q_{i,j}$ ($1\le i\le r$, $1\le j\le d_i$) into a sequence $q_1,\cdots,q_d$ such that $q_{i,j}$ precedes $q_{i',j'}$ if and only if $j<j'$ or $j=j'$ and $i<i'$. In particular, $\deg q_1\le \cdots \le \deg q_d$. Since $X$ is smooth, the linear subspace $\Sigma_x = (q_1=\cdots=q_r=0)$ has codimension $r$ and can be identified with the tangent space of $X$ at $x$.  By the following lemma, we may assume that $d\ge n+1$ throughout the remaining part of the section.

\begin{lem} \label{lem:large index}
Let $X\subseteq\bP^{n+r}$ be a smooth Fano complete intersection of codimension $r$, dimension $n$ and index $s$ $($i.e. $-K_X\sim sH)$. Assume that $s\ge r+1$, then $\lct(X;|H|_\bQ)=1$ where $H$ is the hyperplane class.
\end{lem}

\begin{proof}
The proof is similar to the hypersurface case in \cite{C-hypersurface-lct}. Let $D\sim_\bQ H$ be an effective divisor on $X$. By \cite[Proposition 2.1]{Suzuki-cpi}, there exists a closed subset $Z\subseteq X$ of dimension at most $r-1$ such that $\mult_x D\le 1$ for all $x \in X\backslash Z$. In particular, $(X,D)$ is lc outside $Z$. Let $W\subseteq X$ be a linear space section of $X$ of codimension $r$ that is disjoint from $Z$. Let $H_1,\cdots,H_{r+1}\subseteq X$ be general hyperplane sections containing $W$ and let $\Delta=(1-\epsilon)D+\frac{r}{r+1}(H_1+\cdots+H_{r+1})$ where $0<\epsilon\ll 1$. Then $(X,\Delta)$ is klt outside $Z\cup W$ by construction. As $-(K_X+\Delta)\sim_\bQ (s-r-1+\epsilon)H$ is ample, $\Nklt(X,\Delta)$ is connected by Koll\'ar-Shokurov's connectedness lemma. Since $\mult_W \Delta\ge r$, $W$ is already a non-klt center of $(X,\Delta)$; hence since $Z$ is disjoint from $W$, we deduce that $(X,\Delta)$ is klt along $Z$ and it follows that $(X,(1-\epsilon)D)$ is also klt along $Z$. Hence $(X,D)$ is lc and $\lct(X;|H|_\bQ)\ge 1$. The other direction of the inequality is obvious.
\end{proof}

\begin{defn}
Let $k=\min\{d,n+r-2\}$. The complete intersection $X$ is said to be \emph{Pukhlikov regular} (or simply \emph{P-regular}) at $x$, if for any linear form $h=h(z_*)$ not in the span of $q_1,\cdots,q_r$, the sequence $h,q_1,\cdots,q_k$ is a regular sequence in $\cO_{\bP^{n+r},x}$. We say that $X$ is \emph{P-regular} if it is P-regular at every $x\in X$.
\end{defn}

Note that when $X$ is a Fano complete intersection of index one, this is exactly the P-regularity condition introduced by Pukhlikov \cite{Puk-cpi}. We now state the generality condition we need.

\begin{defn} \label{defn:m-strong reg}
Let $0\le m < n$ be an integer. We say that $X$ is $m$-strongly P-regular if it is P-regular at $x$ and for any linear form $h=h(z_*)$ not in the span of $q_1,\cdots,q_r$, the algebraic set $Z=(h=q_1=\cdots=q_m=0)\cap X$ is irreducible and reduced. We say that $X$ is $m$-\emph{strongly P-regular} if it is $m$-strongly P-regular at every $x\in X$.
\end{defn}

This can be viewed as a generalization of the conditions introduced in \cite[\S 2.1]{Puk-product}. Clearly if $X$ is $m$-strongly P-regular then it is also $m'$-strongly P-regular for all $m'<m$. Note also that if $n\ge 2r+3$ and $X$ is P-regular, then $X$ is at least $r$-strongly P-regular as by Lefschetz hyperplane theorem, the cycle $Z=(h=q_1=\cdots=q_r=0)\cap X$ is automatically irreducible and reduced. 

Assume that the complete intersection $X$ is P-regular at $x$. Following Pukhlikov, we further introduce the following setup. For each $1\le i\le r$ and $1\le j\le d_i$, let
\[f_{i,j}=q_{i,1}+\cdots+q_{i,j}\]
be the truncated equation of the hypersurface $F_i$ at $x$. Let $\bar{f}_1,\cdots,\bar{f}_d$ be the rearrangement of all the $f_{i,j}$ corresponding to $q_1,\cdots,q_d$. Let $e_i=\deg q_i=\deg \bar{f}_i$. Let $1\le \ell \le k:= \min\{d,n+r-2\}$ be an integer and let $u,v$ be such that $q_\ell=q_{u,v}$. We define the $\ell$-th slope $\beta_\ell$ to be $\frac{v+1}{v}$ (note that $v=e_\ell$) if $v<d_u$ and $q_{u,v+1}$ belongs to the sequence $q_1,\cdots,q_k$; otherwise (including when $\ell>k$) we set $\beta_\ell =1$. We also call the linear system of divisors
\[\Lambda_\ell = \{ \sum_{1\le i\le \ell, \beta_i>1} \bar{f}_i s_i\,|\,s_i=s_i(z_*)\;\mathrm{is}\;\mathrm{homogeneous},\,\deg s_i = e_\ell - e_i \} \subseteq |\cO_X(e_\ell)|\]
the $\ell$-th hypertangent linear system of $X$ at $x$. If $L$ is a hyperplane section of $X$ at $x$ that does not contain the tangent space $\Sigma_x$ of $X$, we also refer to $\Lambda_{L,\ell}:=\Lambda_\ell |_L$ as the $\ell$-th hypertangent linear system of $L$ at $x$. By the P-regularity condition, we have $\mult_x \Lambda_{L,\ell} = e_\ell +1$ when $\beta_\ell >1$ and in this case it is not hard to see that as $\Bs(\Lambda_{L,\ell})$ is the intersection of $L$ with all $(q_i=0)$ for which $i\le\ell$ and $\beta_i>1$, the tangent cone of $\Bs(\Lambda_{L,\ell})$ is a complete intersection of codimension $\lambda_\ell:=\#\{i\le \ell \,|\,\beta_i>1\}$ in $\bP(L\cap\Sigma_x)$. Therefore $\codim_L \Bs(\Lambda_{L,\ell})=\lambda_\ell$ in a neighbourhood of $x$ and 
\[\mult_x \Bs(\Lambda_{L,\ell})=\prod_{1\le i\le\ell, \beta_i>1} (e_i+1)=\prod_{1\le i\le\ell, \beta_i>1} e_i \beta_i\]
(for more details see \cite[\S 3]{Puk-book}). If in addition $X$ is $\ell$-strongly P-regular at $x$, then we also have $\Bs(\Lambda_{L,\ell})=L\cap (\bar{f}_1=\cdots=\bar{f}_\ell=0)$ and
\[\deg \Bs(\Lambda_{L,\ell}) = \deg X \cdot \prod_{1\le i\le\ell, \beta_i>1} e_i\]
and hence
\begin{equation} \label{eq:m/d of Z_l}
    \md \Bs(\Lambda_{L,\ell}) = \frac{1}{\deg X} \prod_{1\le i\le\ell, \beta_i>1} \beta_i = \frac{1}{\deg X} \prod_{1\le i\le\ell} \beta_i.
\end{equation}
It is also clear that $\lambda_k=k-r$. 

We can then state the lower bound of log canonical threshold on $m$-strongly P-regular complete intersections.

\begin{lem} \label{lem:lct bound}
Let $m\ge 1$ be an integer such that $\beta_m>1$ and let $X\subseteq\bP^{n+r}$ be an $(m-1)$-strongly P-regular complete intersection, then with the above notations, we have
\[\lct(X;|H|_\bQ)\ge \min\left\{1,\frac{2}{\deg X}\prod_{i\ge 1, i\neq m} \beta_i\right\}.\]
\end{lem}

\begin{proof}
The proof is similar to that of \cite[Theorem 2]{Puk-product}, based on the technique of hypertangent divisors. Denote by $c$ the right hand side of the above inequality. Let $D\sim_\bQ H$ be an effective divisor on $X$. It suffices to show that $(X,cD)$ is lc. Since $c\le 1$, the non-lc center of the pair $(X,cD)$ has codimension at least $2$. Let $x\in X$ be a point in the non-lc center of $(X,cD)$. Let $\pi:\tX\rightarrow X$ be the blowup of $x$, let $E$ be the exceptional divisor and let $\tD$ be the strict transform of $D$. Then by (the proof of) \cite[\S 7, Proposition 2.3]{Puk-book} (applied to a general surface section of $(X,cD)$), there exists a hyperplane $\Pi\subseteq E$ such that
\begin{equation} \label{eq:mult>2}
    \mult_x (cD) + \mult_\Pi (c\tD) >2.
\end{equation}
Let $\cL$ be the linear system of hyperplane sections of $X$ whose strict transform contains $\Pi$ and let $L\in \cL$ be a general member. Note that $\cL$ is spanned by the tangent hyperplanes of $X$ at $x$ and another hyperplane. Thus by the P-regularity condition, $\Bs(\cL)$ has codimension at least $r+1\ge 2$ in $X$ and the divisor $L\cdot D \subseteq L$ is well defined. By \eqref{eq:mult>2} we have $\mult_x(cD\cdot L)>2$, or
\begin{equation} \label{eq:mult/deg}
    \md(D\cdot L)>\frac{2}{c\cdot \deg X}.
\end{equation}
Let $\Lambda_{L,\ell}$ be the $\ell$-th hypertangent linear system of $L$ at $x$. Since $X$ is $(m-1)$-strongly P-regular, the cycle $Z_\ell=\Bs(\Lambda_{L,\ell})$ is irreducible and reduced when $1\le \ell \le m-1$.

By \eqref{eq:mult/deg}, there exists an irreducible component $W_1$ of $D\cdot L$ such that 
\[\md W_1 \ge \md(D\cdot L)>\frac{2}{c\cdot \deg X}\ge \frac{2}{\deg X}\]
and hence by \eqref{eq:m/d of Z_l} with $\ell=1$ we have $W_1\neq Z_1$ (recall that $\beta_1=2$ and indeed we have $\beta_\ell\le 2$ for all $\ell$). It follows that $W_1$ is not contained in the divisor $H_1 = \Lambda_{L,1}$ (assuming $m-1\ge 1$) and we have a well-defined codimension $2$ cycle $W_1\cdot H_1$ in $L$ with
\[\md(W_1\cdot H_1)\ge \frac{2}{1}\cdot \md W_1=\beta_1 \cdot \md W_1\]
as $\mult_x H_1 = 2$ and $H_1\sim H$. Let $W_2$ be an irreducible component of $W_1\cdot H_1$ such that
\[\md W_2\ge\md(W_1\cdot H_1)>\frac{2\beta_1}{c\cdot \deg X} \ge \frac{2\beta_1}{\deg X}.\]
Assume that $\beta_2>1$, then by \eqref{eq:m/d of Z_l} again we have $W_2\neq Z_2$ (assuming $m-1\ge 2$). Since $W_2$ is already contained in $H_1$ and $L$ we deduce that $W_2$ is not contained in a general divisor $H_2\in \Lambda_{L,2}$, and in this case we get a codimension $3$ cycle $W_3$ in $L$ as an irreducible component of $W_2\cdot H_2$ with
\begin{equation} \label{eq:W_2}
    \md W_3 \ge \md(W_2\cdot H_2)\ge \beta_2 \cdot \md W_2 > \frac{2\beta_1 \beta_2}{c\cdot \deg X}.
\end{equation}
If $\beta_2=1$, we simply set $W_3=W_2$ and the inequality \eqref{eq:W_2} clearly still holds.
Iterating this process, we find an irreducible cycle $W_m$ of codimension $\lambda_m=\lambda_{m-1}+1$ (where $\lambda_\ell:=\#\{i\le \ell \,|\,\beta_i>1\}$) in $L$ with
\[\md W_m > \frac{2\beta_1 \beta_2 \cdots \beta_{m-1}}{c\cdot \deg X}.\]
By the P-regularity condition, $Z_\ell$ has codimension $\lambda_\ell$ in a neighbourhood of $x$, hence if $\beta_{m+1}>1$ (so that $\lambda_{m+1}=\lambda_m +1$) then $W_m$ is not contained in $\Bs(\Lambda_{L,m+1})$ and we may repeat the construction above to get an irreducible cycle $W_{m+1}$ (of codimension $\lambda_{m+1}$ in $L$) in the support of $W_m\cdot H_{m+1}$ (where $H_{m+1}$ is a general divisor in $\Lambda_{L,m+1}$) such that
\[\md W_{m+1} > \frac{2\beta_1 \beta_2 \cdots \beta_{m-1} \beta_{m+1}}{c\cdot \deg X}.\]
If $\beta_{m+1}=1$ then we simply set $W_{m+1}=W_m$ and the above inequality still holds for obvious reason. Iterate this process again and eventually we obtain an irreducible cycle $W_k$ (where $k = \min\{d,n+r-2\}$) such that
\[\md W_k > \frac{2\beta_1 \beta_2 \cdots \beta_{m-1} \beta_{m+1}\cdots \beta_k}{c\cdot \deg X}.\]
But we always have $\md W \le 1$ for any subvarieties $Z\subseteq X$, hence
\[c>\frac{2}{\deg X}\prod_{i\ge 1, i\neq m} \beta_i\]
(recall that $\beta_i=1$ when $i>k$), a contradiction. Therefore, $(X,cD)$ is lc at every $x\in X$.
\end{proof}

Our remaining task is to show that a general complete intersection of given degree is $m$-strongly P-regular for suitable $m$ so that we can apply Lemma \ref{lem:lct bound}. First we deal with the P-regularity condition.

\begin{lem} \label{lem:general cpi P-regular}
Let $r\ge 1$, $n\ge 2r+3$ and $d_1,\cdots,d_r\ge 2$ be integers. Then a general complete intersection $X\subseteq \bP^{n+r}$ of hypersurfaces of degrees $d_1,\cdots,d_r$ is P-regular.
\end{lem}

\begin{proof}
This should be well known to experts and is essentially a direct consequence of the work of Pukhlikov \cite[\S 4]{Puk-cpi}. For reader's convenience we sketch the proof. Let $\bP=\bP^{n+r}$. Let $U\rightarrow B\subseteq \prod_{i=1}^r \bP H^0(\bP,\cO_{\bP}(d_i))$ be the universal family of smooth complete intersections of given type and $U\rightarrow\bP$ be the natural projection. It suffices to show that for a smooth complete intersection containing $x$ to violate the P-regularity at $x$ imposes at least $n+1$ conditions on the coefficients of $q_1,\cdots,q_d$ for any $x\in \bP$, since then the locus of irregular points has codimension at least $n+1$ in $U$ and cannot dominate $B$. For this we may fix the coefficients of $q_1,\cdots,q_r$ and further assume that $d\ge n+r-2$ by introducing auxiliary terms in the sequence $q_1,\cdots,q_d$.

Let $p_1,\cdots,p_{n-2}$ be the sequence obtained by restricting $q_{r+1},\cdots,q_{n+r-2}$ to $T_x X=(q_1=\cdots=q_r=0)\cap T_x \bP$ and let $m_i=\deg p_i\ge 2$. Let $\Pi=\bP(T_x X)$. Then the P-regularity of $X$ at $x$ is equivalent to saying that every hyperplane section of $W=(p_1=\cdots=p_{n-2}=0)\subseteq\Pi$ is zero dimensional. There are now two cases to consider. If all the $p_i$ vanish on a line, it would require at least
\begin{equation} \label{eq:contain a line}
    \sum_{i=1}^{n-2} (m_i+1) - 2(n-2) = \sum_{i=1}^r \frac{a_i(a_i+1)}{2} - (n+r-2)
\end{equation}
conditions on the coefficients of $p_i$ (where $a_i$ is the largest subscript $j$ such that $q_{i,j}$ appears in $q_1,\cdots,q_{n+r-2}$). As $\sum_{i=1}^r a_i=n+r-2$, an elementary argument shows that if $n\ge 2r+3$ then this number is $\ge n+1$. On the other hand, for a fixed hyperplane $\Sigma \subseteq \Pi$ we claim that for $W\cap \Sigma$ to contain a component of positive dimension that is not a line would require at least $2n$ conditions on the coefficients of $p_i$. As the hyperplane $\Sigma$ varies in an $(n-1)$-dimensional family, this will prove the statement we want.

Let $p'_i=p_i|_\Sigma$ and let $1\le \ell \le n-2$. To verify the claim, it suffices to consider the case where $p'_1,\cdots,p'_{\ell-1}$ form a regular sequence and $p'_\ell|_B=0$ for some component $B$ of $(p'_1=\cdots=p'_{\ell-1}=0)\subseteq \Sigma$ whose linear span $\left\langle B \right\rangle$ is different from a line. If $m_\ell=2$ then $\ell\le r$ and as $\dim B=n-\ell-1$ we see that (see e.g. Lemma \ref{lem:hyp contain X}) this already imposes at least
\begin{equation} \label{eq:m_l=2}
    \binom{n-\ell+1}{2}\ge \binom{n-r+1}{2}\ge 2n
\end{equation}
conditions on $p'_\ell$ (given that $n\ge 2r+3$). Similarly if $m_\ell =3$ and $\ell\le r+1$, then the vanishing of $p'_\ell$ on $B$ imposes at least
\begin{equation} \label{eq:m_l=3}
\binom{n-\ell+2}{3}\ge \binom{n-r+1}{3}\ge 2n
\end{equation}
conditions on $p'_\ell$. Thus we may assume that either $m_\ell\ge 4$ or $\ell\ge r+2$; in particular, this implies $\sum_{i=1}^\ell m_i\ge 2\ell+2$. Since any nonzero product $\prod_{j=1}^{m_i} \ell_j$ (where each $\ell_j$ is a linear form on $\left\langle B \right\rangle$) does not vanish on $B$ and the set of such products has dimension $m_i(n-2-b)+1$ (where $b=\codim_\Sigma \left\langle B \right\rangle$) in the space of degree $m_i$ homogeneous forms on $\left\langle B \right\rangle$, we deduce that our assumption on $p'_1,\cdots,p'_\ell$ imposes at least
\begin{equation} \label{eq:m_l>2}
    \left(\sum_{i=1}^\ell m_i\right)(n-2-b)+\ell-\dim \mathrm{Gr}(b,n-1) \ge (2\ell+2)(n-2-b)+\ell-b(n-1-b)
\end{equation}
conditions on their coefficients. Since $\codim B\le \ell-1$ and $\left\langle B \right\rangle$ is not a line, we have $b\le \ell-1$ and $b\le n-4$ when $\ell=n-2$. Therefore, it is straightforward to check that the right hand side of \eqref{eq:m_l>2}, as a quadratic function in $b$, in minimized at $b=\ell-1$ when $\ell\le n-3$ or at $b=n-4$ when $\ell=n-2$. It then follows from another elementary argument that the right hand side of \eqref{eq:m_l>2} is $\ge 2n$.
\end{proof}

For later applications we also need the following slightly more general version.

\begin{lem} \label{lem:cone cpi P-regular}
Let $n,r\ge 1$, $0\le s<r$ and $d_1,\cdots,d_r\ge 2$ be integers such that
\[n\ge 2r+3+\max\left\{2\sum_{i=1}^s d_i , \frac{1}{2}\sum_{i=1}^s d_i(d_i+1)\right\}.\]
Let $y\in\bP^{n+r}$, let $Q_1,\cdots,Q_s$ be general hypersurfaces of degrees $d_1,\cdots,d_s$ with a cone singularity at $y$ $($i.e. $\mult_y Q_i=d_i)$ and let $Q_{s+1},\cdots,Q_r$ be general hypersurfaces of degrees $d_{s+1},\cdots,d_r$. Then the complete intersection $X=Q_1\cap\cdots\cap Q_r\subseteq \bP^{n+r}$ is P-regular.
\end{lem}

\begin{proof}
As before we assume $d=\sum_{i=1}^r d_i \ge n+r-2$. Since $s<r$, a general such complete intersection is smooth and does not contain $y$. For each $x\neq y\in \bP^{n+r}$, we may choose affine linear coordinates $z_*=(z_1,\cdots,z_{n+r})$ with origin at $x$ such that a hypersurface has a cone singularity at $y$ if and only if its equation only involves the variables $z_2,\cdots,z_{n+r}$. After a further change of variable, we may also assume that $T_x X$ (throughout the proof we use the same notations as in the proof of Lemma \ref{lem:general cpi P-regular}) is defined by the vanishing of some of the $z_i$'s. As before, it suffices to show that violation of P-regularity condition imposes at least $n+1$ conditions on the coefficients of the $p_i$. If $z_1$ vanishes on $T_x X$, then the restricted expression $p_i$ can be an arbitrary degree $m_i$ homogeneous polynomial and the same calculation in Lemma \ref{lem:general cpi P-regular} applies. If $z_1$ does not vanish on $T_x X$, the only restriction is that when $p_i$ comes from some $Q_j$ with $j\le s$, it can only vary among the equations of (arbitrary) degree $m_i$ hypersurfaces with a cone singularity at $w=[1:0:\cdots:0]\in \Pi$.

The proof then proceeds as before. For instance (this is the most complicated case), if $\Sigma\subseteq\Pi$ is a hyperplane containing $w$ and $B\subseteq \Sigma$ is a positive dimensional subvariety whose linear span $\left\langle B \right\rangle$ also contains $w$ but is different from a line, then requiring $p'_1,\cdots,p'_\ell$ ($m_\ell\ge 3$) to vanish on $B$ imposes at least 
\[\left(\sum_{i=1}^\ell m_i\right)(n-2-b)+\ell-\dim \mathrm{Gr}(b,n-2)- \sum\nolimits' m_i \ge 3n-5-\sum_{i=1}^s \frac{1}{2}d_i(d_i+1)\ge 2n\]
conditions on the coefficients of $p'_1,\cdots,p'_\ell$ (where $b$ is the codimension of $\left\langle B \right\rangle$ in $\Sigma$ and $\sum'$ sums over all $1\le i\le \ell$ such that $p_i$ comes from some $Q_j$ with $1\le j\le s$; clearly $\sum' m_i\le \sum_{i=1}^s \frac{1}{2}d_i(d_i+1)$). On the other hand, requiring all $p_i$ to contain a line passing through $w$ imposes at least $m_i+1$ (resp. $1$) on coefficients of $p_i$ if $p_i$ comes from the equation of $Q_j$ for some $j>s$ (resp. $j\le s$), and we get a total of at least
\begin{align*}
    \sum_{i=s+1}^r \frac{1}{2}a_i (a_i+1) - (r-s)  & \ge \sum_{i=s+1}^r \frac{1}{2}a_i (a_i+1) - r \\
    & \ge 2\sum_{i=s+1}^r a_i -r \ge 2(n+r-2-\sum_{i=1}^s d_i) - r \ge n+1
\end{align*}
conditions (to see the second inequality, note that $\sum_{i=s+1}^r a_i \ge n+r-2-\sum_{i=1}^s d_i\ge 3r$ and the quadratic form $\sum a_i (a_i+1)$ is minimized when the differences between the $a_i$'s are at most $1$, hence we may assume $a_i\ge 3$ for all $s+1\le i\le r$). The remaining cases can be treated in a similar and easier fashion as in Lemma \ref{lem:general cpi P-regular}.
\end{proof}

Recall that a smooth P-regular complete intersection $X\subseteq \bP^{n+r}$ of dimension $n\ge 2r+3$ is automatically $r$-strongly P-regular, we obtain the following immediate corollary in the spirit of \cite{Puk-cpi}.

\begin{cor} \label{cor:CY cpi}
Let $X\subseteq\bP^{n+r}$ be a smooth Calabi-Yau complete intersection of dimension $n$ by hypersurfaces of degrees $d_1\le\cdots\le d_r$. Assume that $n\ge 2r+3$, $d_r\ge 12$ and $X$ is P-regular. Then $\lct(X;|H|_\bQ) = 1$ where $H$ is the hyperplane class.
\end{cor}

\begin{proof}
Since $X$ is CY, $d=\sum_{i=1}^r d_i = n+r+1$, hence by the definition of $\beta_i$, the product $\prod_{i\ge1} \beta_i = \beta_1\beta_2\cdots\beta_{n+r-2}$ is obtained by removing three smallest factors (and then adding $r$ factors of $1$) from the following expression
\[\frac{2}{1}\cdot \frac{3}{2}\cdot \ldots \cdot \frac{d_1}{d_1-1}\cdot \frac{2}{1}\cdot \ldots \cdot \frac{d_2}{d_2-1}\cdot \ldots \cdot \frac{2}{1}\cdot \ldots \cdot \frac{d_r}{d_r-1},\]
which equals $\deg X = d_1 d_2\cdots d_r$. As $d_r\ge 12$, we get
\begin{equation} \label{eq:prod of beta}
    \prod_{i\ge1} \beta_i \ge \frac{d_r-3}{d_r-2}\cdot \frac{d_r-2}{d_r-1}\cdot \frac{d_r-1}{d_r} \cdot \deg X \ge \frac{3}{4} \deg X.
\end{equation}
By rearranging the sequence $q_1,\cdots,q_d$ we may assume that the first quadratic term $q_{r+1}$ comes from the degree $d_r$ equation and hence $\beta_{r+1}=\frac{3}{2}>1$. As $X$ is $r$-strongly P-regular under our assumptions, we have
\[\lct(X;|H|_\bQ)\ge \min\left\{1,\frac{2}{\beta_{r+1}\deg X}\prod_{i\ge1} \beta_i\right\} = \min\left\{1,\frac{4}{3\deg X}\prod_{i\ge1} \beta_i\right\}\ge 1\]
by Lemma \ref{lem:lct bound} and \eqref{eq:prod of beta}.
\end{proof}

For the construction of K-unstable Fano varieties, we will need the special case of codimension $2$ complete intersections where one of the defining hypersurfaces has a cone singularity.

\begin{cor} \label{cor:codim 2 CY}
Let $e\ge 2$, $n\ge 8+ \max\{2e,\frac{1}{2}e(e+1)\}$ be integers and let $f(x_0,\cdots,x_n)$ and $g(x_0,\cdots,x_{n+1})$ be homogeneous polynomials of degrees $e$ and $n+2-e$, respectively. Assume that the coefficients of $f$ and $g$ are general and let $X=(f=g=0)\subseteq\bP^{n+1}$, then $\lct(X;|H|_\bQ) = 1$ where $H$ is the hyperplane class.
\end{cor}

\begin{proof}
By Lemma \ref{lem:cone cpi P-regular}, $X$ is P-regular. By assumption, $n+2-e\ge 12$ and $X$ is Calabi-Yau. Hence the result follows directly from Corollary \ref{cor:CY cpi}.
\end{proof}

When the dimension $n$ is not too small, Pukhlikov essentially proves that a general hypersurface of degree $d\ge n+1$ is $2$-strongly P-regular. Combining this with Lemma \ref{lem:lct bound} we obtain:

\begin{cor} \label{cor:n+1/d}
Let $n\ge 5$ and let $X\subseteq\bP^{n+1}$ be a general hypersurface of degree $d\ge n+1$, then $\lct(X;|H|_\bQ)\ge \frac{n+1}{d}$.
\end{cor}

\begin{proof}
Consider the following variant of the $m$-strongly P-regular condition on complete intersections (we use the same notation as in Definition \ref{defn:m-strong reg}):
    \begin{enumerate}
        \item[$(I_m)$] For any linear form $h=h(z_*)$ not in the span of $q_1,\cdots,q_r$, the algebraic set $Z=(h=q_1=\cdots=q_m=0)\cap X$ is equidimensional and for any irreducible component $Z_i$ of $Z$, we have \[\md(Z)=\md(Z_i).\]
    \end{enumerate}
When $m=2$ and $X$ is a hypersurface, this condition is implied by the conditions (R1.2)-(R1.3) of \cite[\S 2.1]{Puk-product}, thus by Lemma \ref{lem:general cpi P-regular} and \cite[Proposition 4]{Puk-product}, a general hypersurface of degree $d\ge n+1$ is P-regular and satisfies the condition $(I_2)$ when $n\ge 5$ (although \cite[Proposition 4]{Puk-product} is only stated for degree $n+1$ hypersurfaces, the same proof works in general). It is also not hard to see that the proof of Lemma \ref{lem:lct bound} applies without change to P-regular complete intersections that satisfy $(I_{m-1})$. In particular, taking $m=3$ in Lemma \ref{lem:lct bound}, we have $\lct(X;|H|_\bQ)\ge \frac{3(n-1)}{2d} \ge \frac{n+1}{d}$ as desired.
\end{proof}

Again, the case when $X$ is a Calabi-Yau hypersurface is of particular importance in our construction of weakly exceptional singularities and strictly K-semistable Fano variaties with largest alpha invariant.

\begin{cor} \label{cor:CY hypersurf}
Let $X\subseteq \bP^{n-1}$ be a general hypersurface of degree $n$ where $n=4$ or $n\ge 7$, then 
\[\lct(X;|H|_\bQ)\ge \frac{n-1}{n}.\]
\end{cor}

\begin{proof}
When $n=4$, the result follows from \cite[Theorem 1.2 and Lemma 3.2]{ACS-conj-of-tian}. If $n\ge 7$, this is a special case of Corollary \ref{cor:n+1/d}.
\end{proof}

\begin{cor} \label{cor:weakly exceptional}
Let $f=f(x_1,\cdots,x_n)$ be a general homogeneous polynomials of degree $n$ where $n=4$ or $n\ge 7$, then the singularity $0\in (f(x_1,\cdots,x_n)+x_{n+1}^{n+1}=0)\subseteq \bA^{n+1}$ is weakly exceptional.
\end{cor}

\begin{proof}
Denote the given singularity by $(0\in X)$. The weighted blow-up of $\bA^{n+1}$ with weights $(n+1,\cdots,n+1,n)$ induces a plt blow-up $\pi:(E\subseteq Y)\rightarrow(0\in X)$ with Koll\'ar component $E$. The corresponding log Fano pair $(E,\Diff_E(0))$ is isomorphic to (c.f. \cite[Example 2.4]{K-plt-blowup})
\[\left( (f(x_1,\cdots,x_n)+x_{n+1}=0)\subseteq \bP(1,\cdots,1,n), \frac{n}{n+1}(x_{n+1}=0)\right) = \left( \bP^{n-1},\frac{n}{n+1}S \right)\]
where $S\subseteq\bP^{n-1}$ is the hypersurface defined by $f(x_1,\cdots,x_n)$. By Theorem \ref{thm:alpha and weakly exceptional}, our statement is equivalent to
\begin{equation} \label{eq:alpha>=1}
    \alpha\left( \bP^{n-1},\frac{n}{n+1}S \right)\ge 1.
\end{equation}
To see this, let $D\subseteq \bP^{n-1}$ be an effective divisor such that $D\sim_\bQ H \sim_\bQ -\frac{n+1}{n}(K_{\bP^{n-1}}+\frac{n}{n+1}S)$, we need to show that $(\bP^{n-1},\frac{n}{n+1}(S+D))$ is lc. We may assume that $D$ is irreducible since being lc is preserved under convex linear combinations. If $D$ is supported on $S$ then $D=\frac{1}{n}S$ and the result is clear. If $D$ is not supported on $S$ then by Corollary \ref{cor:CY hypersurf}, the pair $(S,\frac{n-1}{n}D|_{S})$ is lc, hence by inversion of adjunction, $(\bP^{n-1},S+\frac{n-1}{n}D)$ is lc. Since $\lct(\bP^{n-1};|H|_\bQ)=1$, the pair $(\bP^{n-1},D)$ is also lc. By taking convex linear combinations we see that $(\bP^{n-1},\frac{n}{n+1}(S+D))$ is lc as well since
\[\frac{n}{n+1}(S+D)=\frac{n}{n+1}\left( S+\frac{n-1}{n}D \right) + \frac{1}{n+1} D. \]
This finishes the proof.
\end{proof}

\begin{rem} \label{rem:gap}
It is not true that $\alpha\left( \bP^{n-1},\frac{n}{n+1}S \right)\ge 1$ for every smooth hypersurface $S\subseteq \bP^{n-1}$ of degree $n$. For example, let $S$ be a hypersurface with an Eckardt point at $x=[1:0:\cdots:0]$, e.g. $S=(x_2 g(x_1,\cdots,x_n)+h(x_3,\cdots,x_n)=0)$ (where $x_1,\cdots,x_n$ are the homogeneous coordinates of $\bP^{n-1}$) and let $D=(x_2=0)$, then it is not hard to see that $(\bP^{n-1},\frac{n}{n+1}(S+D))$ is not lc by considering the weighted blow-up at $x$ with weights $(w_2,\cdots,w_n)=(n,1,\cdots,1)$. In particular, the examples of \cite[Example 2.4]{K-plt-blowup} and \cite[Example 4.19]{P-plt-blowup} (which are claimed to be weakly exceptional there) are actually not weakly exceptional.
\end{rem}

Our next goal is to show that a general complete intersection (of given codimension) is $m$-strongly P-regular when $n\gg0$ so that we can apply Lemma \ref{lem:lct bound} to a larger range of complete intersections. This intuitively clear result is not needed in the rest of the paper but might be of independent interest. Readers who are mainly interested in the construction of strictly K-semistable and K-unstable Fano varieties with suitably large alpha invariants may skip this part and proceed to the next section.

The proof of $m$-strong P-regularity is done in several steps, spreading from Lemma \ref{lem:hyp contain X} to \ref{lem:m-reg}. The idea is to introduce a stronger condition by requiring the complete intersection $(h=q_1=\cdots=q_m=q_{1,d_1}=\cdots=q_{r,d_r}=0)$ (recall that each $q_i$ or $q_{i,j}$ is a homogeneous term of the local expression of one of the defining equations of $X$) to be normal (i.e. smooth in codimension $1$). We then show that the number of conditions imposed by the failure of this property at a given point $x\in \bP^{n+r}$ is at least quadratic in $n$, thus since the point $x$ only varies in an $(n+r)$-dimensional family, a general complete intersection will satisfy the strong P-regularity as we want.

\begin{lem} \label{lem:hyp contain X}
Let $X\subseteq \bP^n$ be a quasi-projective variety of dimension $m$. Let
\[\cS=\bP H^0(\bP^n,\cO_{\bP^n}(d))\]
be the space of all hypersurface of degree $d$ and $\cS_X$ the subset of hypersurfaces containing $X$, then
\[\codim_\cS \cS_X\ge\binom{m+d}{d}.\]
\end{lem}

\begin{proof}
Clearly $\cS_X$ is a linear subspace of $\cS$. Let $\phi:\bP^n\dashrightarrow \bP^m$ be a general linear projection whose restriction to $X$ has dense image, and let $\cN=\phi^{-1} |\bP^m,\cO_{\bP^m}(d)|\subseteq \cS$, then by construction none of the hypersurfaces in $\cN$ contain $X$, thus $\cN\cap \cS_X=\emptyset$ and we have $\codim_\cS \cS_X\ge \dim\cN+1= \binom{m+d}{d}$.
\end{proof}

\begin{lem} \label{lem:cpi codim at most c}
Let $c\le q$ and $d_1\le \cdots \le d_q$ be positive integers. Let $X\subseteq \bA^n_{x_1,\cdots,x_n}$ be a quasi-projective variety $($i.e. it's locally closed in $\bA^n)$ of dimension $m\ge c$ and let $f_1,\cdots,f_q$ be regular functions on $X$. Let $\cT$ be the set of $g_i\in k[x_1,\cdots,x_n]_{\le d_i}$ $(i=1,\cdots,q)$ such that $Z=Z(f_1+g_1,\cdots,f_q+g_q)$ has codimension at most $c$ in $X$, then
\[\codim_\cS \cT\ge \sum_{i=1}^{q-c} \binom{m-c+d_i}{d_i}\]
where $\cS=\prod_{i=1}^q k[x_1,\cdots,x_n]_{\le d_i}$.
\end{lem}

\begin{proof}
Let $\cT_j \subseteq \cS$ be the set of $(g_i)_{1\le i \le q}$ such that $Z$ has codimension exactly $j$. By induction on $c$ it suffices to show that
\begin{equation} \label{eq:codim T_c}
    \codim_\cS \cT_c\ge \sum_{i=1}^{q-c} \binom{m-c+d_i}{d_i}.
\end{equation}
Let $\sigma=(i_1,\cdots,i_c)\subseteq\{1,2,\cdots,q\}$ be an index set and let $\cT_{c,\sigma}\subseteq \cS$ be the set of $(g_i)$ such that $Z_\sigma=\cap_{i\in \sigma} Z(f_i+g_i)$ has codimension $c$ in $X$, then clearly
\begin{equation} \label{eq:T=union}
    \cT_c\subseteq \bigcup_{|\sigma|=c} \cT_{c,\sigma}.
\end{equation}
Fix an index set $\sigma=(i_1,\cdots,i_c)$ and a subsequence $(g_j)_{j\in \sigma}$ such that $Z_\sigma=\cap_{i\in \sigma} Z(f_i+g_i)$ has codimension $c$ in $X$, then $(g_1,\cdots,g_q)\in \cT_c$ if and only if $f_i+g_i$ vanishes on some irreducible component of $Z_\sigma$ for each $i\not\in \sigma$. Then set of such $(g_i)_{i\not\in\sigma}$ is either empty or (after translating by an element in this set) isomorphic to the set of $g_i$ that vanishes on some irreducible component of $Z_\sigma$. As $Z_\sigma$ is a complete intersection of dimension $m-c$, this imposes at least $\binom{m-c+d_i}{d_i}$ conditions on $g_i$ by Lemma \ref{lem:hyp contain X} (we identify $k[x_1,\cdots,x_n]_{\le d}$ with $H^0(\bP^n,\cO_{\bP^n}(d))$). Since this holds for every $(g_j)_{j\in \sigma}$, we get
\[\codim_\cS (\cT_{c,\sigma}\cap \cT_c)\ge \sum_{i\not\in\sigma} \binom{m-c+d_i}{d_i}.\]
Combining this with \eqref{eq:T=union} and $d_1\le\cdots\le d_q$, we obtain the desired inequality \eqref{eq:codim T_c}.
\end{proof}


\begin{lem} \label{lem:sing in codim c}
Given integers $m,d,c\ge 1$, then for any integer $n\ge m+c$ and any smooth quasi-projective variety $X\subseteq \bP^n$ of codimension $m$ we have $\codim_\cS(\cS\backslash\cT)\ge p(n)$ where
\[p(n) = \min\left\{\binom{n-m+d}{d},\left(\frac{n-m-1}{2}-c\right)\binom{n-m-c+d-1}{d-1}\right\}\]
and $\cS=\bP H^0(\bP^n,\cO_{\bP^n}(d))$ is the space of all hypersurface of degree $d$ and $\cT$ is the subset of degree $d$ hypersurfaces $D$ such that $X\cap D\subsetneqq X$ is smooth in codimension $c$.
\end{lem}

Note that $p(n)$ grows like $C n^d$ for some constant $C>0$ when $n\gg 0$.

\begin{proof}
Let $x\in X$. Since $X$ is smooth, we can find linear affine coordinates $x_1,\cdots,x_n$ at $x$ such that $\rd x_i$ ($i=m+1,\cdots,n$) generate $\Omega_X^1$ in a neighbourhood of $x$. We may cover $X$ by finitely many open subset with this property and it suffices to prove the lemma for each of these. Thus we may assume $X\subseteq\bA^n$ is locally closed with coordinates $x_1,\cdots,x_n$ chosen as above. Note that up to projectivization, $\cS$ can be identified with $k[x_1,\cdots,x_n]_{\le d}$, the space of degree at most $d$ polynomials. Also note that we may assume $n\gg0$ in proving the statement. By Lemma \ref{lem:hyp contain X}, the set of hypersurfaces $D$ in $\cS$ that contains $X$ has codimension at least $\binom{n-m+d}{d}\ge p(n)$. Thus it suffices to find $p(t)$ such that requiring $D\cap X$ to be singular in codimension $c$ imposes at least $p(n)$ conditions on the defining equation of $D$.

Now let $q=\lfloor \frac{n-m}{2} \rfloor$ and consider 
\[f=f_0+\sum_{i=1}^q x_{m+i}g_i \in k[x_1,\cdots,x_n]_{\le d}\]
where $f_0 \in k[x_1,\cdots,x_n]_{\le d}$ and $g_i \in k[x_{n-q+1},\cdots,x_n]_{\le d-1}$. If we can show that requiring $Z(f)\cap X$ to be singular in codimension $c$ imposes at least $p(n)$ conditions on the $(q+1)$-tuple $(f_0,g_1,\cdots,g_q)$ for some degree $d$ polynomial $p$, then we are done as then for each choice of $g_1,\cdots,g_q$ we would need the same number of $\ge p(n)$ conditions on $f_0$. For this it suffices to show that for each $f_0$, requiring $Z(f)\cap X$ to be singular in codimension $c$ imposes at least $p(n)$ conditions on $(g_1,\cdots,g_q)$. 

By the Jacobian criterion $Z(f)\cap X$ is singular at $x\in X$ if and only if $\rd f(x)=0\in \Omega^1_X \otimes k(x)$. Since $\rd x_i$ ($i=m+1,\cdots,n$) is a global basis of $\Omega_X^1$, we may rewrite $\rd f=\rd f_0 + \sum_{i=1}^q (g_i \rd x_{m+i} + x_{m+i}\rd g_i)$ as a linear combinations of $\rd x_i$ ($i=m+1,\cdots,n$):
\[\rd f = \sum_{i=1}^q (f_i+g_i) \rd x_{m+i} + \sum_{j=m+q+1}^n h_j \rd x_j\]
where $f_i$ and $h_j$ are regular functions on $X$ and as $\frac{\partial g_i}{\partial x_j}=0$ when $1\le i\le q$ and $1\le j \le m+q$, the $f_i$'s do not depend on $g_i$ (they only depend on $\rd f_0$ and the linear relatioins among $\rd x_i$). It then follows that the singular locus of $Z(f)\cap X$ is contained in the zero locus of $f_i+g_i$ ($i=1,\cdots,q$), hence $Z(f_1+g_1,\cdots,f_q+g_q)\subseteq X$ would have codimension at most $c$ if $Z(f)\cap X$ is singular in codimension $c$. By Lemma \ref{lem:cpi codim at most c}, this imposes at least $(q-c)\binom{n-m-c+d-1}{d-1}\ge \left(\frac{n-m-1}{2}-c\right)\binom{n-m-c+d-1}{d-1}$ conditions on $(g_1,\cdots,g_q)$. 
\end{proof}

\begin{lem} \label{lem:m-reg}
Given integers $m,r\ge 1$ and let $n\gg0$. Then for any $2\le d_1\le \cdots \le d_r$, a general complete intersection $X\subseteq \bP^{n+r}$ by hypersurfaces of degrees $d_1,\cdots,d_r$ is $m$-strongly P-regular.
\end{lem}

\begin{proof}
By Lemma \ref{lem:general cpi P-regular}, a general such complete intersection is P-regular, hence we only need to verify the irreducibility and reducedness of $Z=(h=q_1=\cdots=q_m=0)\cap X$ in Definition \ref{defn:m-strong reg}. As noted before, we may assume $m>r$ (otherwise by Lefschetz hyperplane theorem it suffices to take $n\ge 2r+3$). As in the proof of P-regularity, it suffices to show that for any $x\in\bP^{n+r}$ and any choice of the linear terms $h,q_1,\cdots,q_r$, the failure of the irreducibility or reducedness of $Z$ imposes at least $2n$ conditions on the coefficients of the remaining $q_i$. Let $p_1,\cdots,p_\ell$ ($\ell\le m$) be the sequence obtained by removing terms that appear more than once from the sequence $q_{r+1},\cdots,q_m,q_{1,d_1},\cdots,q_{r,d_r}$ and then restricting to the linear subspace $\Sigma=(h=q_1=\cdots=q_r=0)=(h=0)\cap \bP(T_x X)\cong\bP^{n-2}$. Note that $W=(p_1=\cdots=p_\ell=0)\cap \Sigma$ can be identified with $Z\cap H_\infty$ where $H_\infty$ is the hyperplane at infinity and we may assume that $W$ has codimension $\ell$ by a similar proof of P-regularity, hence if $W$ is irreducible and reduced, then the same holds for $Z$. Therefore, it suffices to show that for $W$ to be reducible or non-reduced imposes at least $2n$ conditions on the coefficients of $p_1,\cdots,p_\ell$. Let $m_i=\deg p_i\ge 2$. 

Assume from now on $n\gg 0$. Apply Lemma \ref{lem:sing in codim c} to $W_0=\Sigma$ we see that requiring $W_1=W_0\cap (p_1=0)$ to be singular in codimension $m$ imposes more than $2n$ conditions on $p_1$ (the number of conditions is at least a polynomial of degree $2$ in $n$). Fix a choice of $p_1$ such that $W_1$ is smooth in codimension $m$. By Lemma \ref{lem:sing in codim c} applied to the smooth locus of $W_1$, we again need more than $2n$ conditions on $p_2$ for $W_2=W_1\cap (p_2=0)$ to be singular in codimension $m-1$. Continue in this way, we see that requiring $W=W_i$ ($1\le i\le \ell$) to be singular in codimension $m+1-i\ge 1$ imposes more than $2n$ conditions on the coefficients of $p_1,\cdots,p_i$. Since each $W_i$ is already a complete intersection, they are also normal if they are smooth in codimension $1$; an easy induction using \cite[Corollary 7.9]{Hartshorne} then shows that all the $W_i$ are connected, normal and therefore integral. In particular, $W=W_\ell$ is integral. In other words, asking $W$ to be reducible or non-reduced imposed more than $2n$ conditions on coefficients of $p_1,\cdots,p_\ell$ and we are done.
\end{proof}

Combining this with Lemma \ref{lem:lct bound} in the hypersurface case, we can now give the proof of Theorem \ref{thm:lct-general-hyp}.

\begin{proof}[Proof of Theorem \ref{thm:lct-general-hyp}]
When $d\le n$, the result follows from \cite[Theorem 1.3]{C-hypersurface-lct}. If $d\ge n+1$, then choose sufficiently large integer $m$ such that $\frac{2(m+1)}{(2-\epsilon)(m+2)} > 1$, by Lemma \ref{lem:m-reg} a general hypersurface $X\subseteq\bP^{n+1}$ of degree $d$ is $m$-strongly P-regular when $n\gg 0$, thus by Lemma \ref{lem:lct bound} we have
\[\lct(X;|H|_\bQ)\ge\min\left\{1,\frac{m+1}{m+2}\cdot \frac{2(n-1)}{d}\right\}\ge \min\left\{1,\frac{2(m+1)}{(2-\epsilon)(m+2)}\cdot \frac{n-1}{n}\right\}\ge 1\]
and the result follows since we always have $\lct(X;|H|_\bQ)\le 1$.
\end{proof}

\section{Main constructions}\label{sec:construction}

\subsection{Strictly K-semistable example}

In this section, we construct singular Fano varieties $X$ of dimension $n\gg 1$ that are strictly K-semistable and $\alpha(X)=\frac{n}{n+1}$.

By Theorem \ref{thm:fujalpha}, every such Fano variety has at least one weakly exceptional singularity $x$ whose corresponding (unique) Koll\'ar component has log discrepancy $n$. One such weakly exceptional singularity is given by Corollary \ref{cor:weakly exceptional}, with corresponding Koll\'ar component $F\cong \bP^{n-1}$ and different $\Delta_F=(1-\frac{1}{n+1})S$ where $S=(f(x_1,\cdots,x_n)=0)$ is a general smooth hypersurface of degree $n$ in $\bP^{n-1}$. We first construct the K-polystable degeneration as in Corollary \ref{cor:polystable degeneration}. Let $M:=\frac{n}{n+1}S-(n-1)H$ be a $\bQ$-divisor on $\bP^{n-1}$ where $H$ is a hyperplane in $\bP^{n-1}$. Let $X_0$ be the projective orbifold cone $C_p(\bP^{n-1},M)$.

\begin{lem} \label{lem:X_0}
Let $n=4$ or $n\ge 7$. Then the projective orbifold cone $X_0$ is a K-polystable Fano variety. Moreover, $X_0$ is isomorphic to the hypersurface $(x_0 f(x_1,\cdots,x_n)+x_{n+1}^{n+1}=0)$ in $\bP^{n+1}$ and under this isomorphism, $\bP^{n-1}_\infty$ can be identified with the locus $(x_0=x_{n+1}=0)$.
\end{lem}

\begin{proof}
Since $\alpha(\bP^{n-1},(1-\frac{1}{n+1})S)\geq 1$ for a general smooth hypersurface $S\subset\bP^{n-1}$ of degree $n$ by Corollary \ref{cor:weakly exceptional}, we know that $(\bP^{n-1},(1-\frac{1}{n+1})S)$ is K-stable. Hence the K-polystability of $X_0=C_p(\bP^{n-1},M)$ follows from Proposition \ref{prop:kpscone} and the fact that $-K_{\bP^{n-1}}-(1-\frac{1}{n+1})S\sim_{\bQ}n M$.
By definition, we have
\[
X_0=\Proj\bigoplus_{m=0}^{\infty}\bigoplus_{i=0}^{\infty} H^0(\bP^{n-1},\cO_{\bP^{n-1}}(\lfloor m M\rfloor))\cdot s^i.
\]
Denote by $R_{m,i}:=H^0(\bP^{n-1},\cO_{\bP^{n-1}}(\lfloor mM\rfloor))\cdot s^i$ and $R:=\oplus_{m=0}^\infty\oplus_{i=0}^\infty R_{m,i}$.
Let $L$ be the coherent sheaf $\cO(n+1)$ on $X_0$ with respect to the grading $R_{m,i}\mapsto m+i$. 
Then $L\cong \cO_{X_0}((n+1)\bP_{\infty}^{n-1})$; since $(n+1)M$ is Cartier on $\bP^{n-1}$, we also see that $L$ is an ample line bundle on $X_0$. We will choose sections $s_0,s_1,\cdots,s_{n+1}$ of $L$ to obtain the desired embedding from $X_0$ to $\bP^{n+1}$. 

To begin with, let us choose $s_0:=1\cdot s^{n+1}\in R_{0,n+1}$. We also fix a generator $z$ of $H^0(\bP^{n-1},\cO_{\bP^{n-1}}(S-nH))$. Since $\lfloor nM\rfloor=(n-1)(S-nH)$, we choose $s_{n+1}:=z^{n-1}\cdot s\in R_{n,1}$.
Let us choose sections $y_1,\cdots,y_n\in H^0(\bP^{n-1},\cO_{\bP^{n-1}}(H))$ that correspond to projective coordinates $x_1,\cdots,x_n$ of $\bP^{n-1}$. Then $y_i z^n$ is a section of $\cO_{\bP^{n-1}}(nS-(n^2-1)H)=\cO_{\bP^{n-1}}(\lfloor(n+1)M\rfloor)$. For each $1\leq i\leq n$, we choose $s_i:=y_i z^n\cdot 1\in R_{n+1,0}$. 

To prove that $[s_0:\cdots:s_{n+1}]$ induces the embedding from $X_0$ to $\bP^{n+1}$, we first show that the linear system generated by $\{s_0,\cdots,s_{n+1}\}$ is base point free. Let $I\subset R$ be the graded ideal generated by $\{s_0,\cdots,s_{n+1}\}$. Thus $s\in \sqrt{I}$ since $s_{0}=s^{n+1}\in I$. Thus $R_{m,i}=R_{m,0}\cdot s^i\subset \sqrt{I}$ whenever $i>0$. If $m>0$ and $i=0$, then any element in $R_{(n+1)m,0}$ is a polynomial in $s_1,\cdots,s_n$. Thus $R_{(n+1)m,0}\subset I$ which implies $R_{m,0}\subset \sqrt{I}$ for any $m>0$. As a result, the base point freeness of the linear system generated by $\{s_0,\cdots,s_{n+1}\}$ follows from $R_+=\sqrt{I}$. Let us denote the induced morphism by $\phi: X_0\to \bP^{n+1}$. 

Let $X_0'$ be the hypersurface in $\bP^{n+1}$ defined by $(x_0 f(x_1,\cdots,x_n)+x_{n+1}^{n+1}=0)$.
We will show that the image $\phi(X_0)$ is contained in $X_0'$, i.e. $s_0 f(s_1,\cdots,s_n)+s_{n+1}^{n+1}=0$. It is clear that 
\[
s_{n+1}^{n+1}=z^{n^2-1}\cdot s^{n+1},\quad s_0 f(s_1,\cdots,s_n)= f(y_1,\cdots,y_n)z^{n^2}\cdot s^{n+1}.
\]
Both terms are considered inside $R_{n(n+1),n}=H^0(\bP^{n-1},\cO_{\bP^{n-1}}(n(n+1)M))\cdot s^{n+1}$.
Meanwhile, $z^{n^2-1}$ is identified with a section of $\cO_{\bP^{n-1}}((n+1)\lfloor nM \rfloor)$ where $n(n+1)M=(n+1)\lfloor nM \rfloor+S$. Hence we may assume that $-f(y_1,\cdots,y_n)z$ is the canonical section of $\cO_{\bP^{n-1}}(S)$ using which the product structure of $R$ is defined. Therefore, 
\[
s_{n+1}^{n+1}=z^{n^2-1}(-f(y_1,\cdots,y_n)z)\cdot s^{n+1}=-s_0 f(s_1,\cdots,s_n).
\]

So far we have constructed a morphism $\phi:X_0\to X_0'$. It remains to show that $\phi$ is an isomorphism. Since $L$ is ample, $\phi$ is finite. We also know that $(\cO_{X_0'}(1)^n)=n+1$ and 
\[
(L^n)=(n+1)^n(\bP^{n-1}_\infty)^n=(n+1)^n(M^{n-1})=n+1, 
\]
hence $\phi$ is birational. It is easy to check that the singular locus of $X_0'$ is the point $[1:0:\cdots:0]$ union $(x_0=x_{n+1}=f=0)$, hence $X_0'$ is normal. Thus the Zariski main theorem implies that $\phi$ is an isomorphism. It is clear from the construction that $\phi^*(x_0=x_{n+1}=0)=\bP^{n-1}_\infty$ as $\phi^*(x_0=0)=(s^{n+1}=0)=(n+1)\bP^{n-1}_\infty$. We finish the proof.
\end{proof}

The strictly K-semistable examples as referred to in Theorem \ref{thm:main} are now obtained by deforming the above K-polystable examples.

\begin{prop} \label{prop:strict K-ss}
Let $n=4$ or $n\ge 7$. Let $f(x_1,\cdots,x_n)$ and $g(x_1,\cdots,x_{n+1})$ be general homogeneous polynomials of degree $n$ and $n+1$ respectively. Let $X\subseteq\bP^{n+1}$ be the hypersurface defined by the equation $(x_0f+g=0)$. Then $\alpha(X)=\frac{n}{n+1}$ and $X$ is strictly K-semistable. In addition, $X$ has a unique singular point at $[1:0:\cdots:0]$.
\end{prop}

\begin{proof}
Since $f$ and $g$ are general, $X$ has a unique isolated singularity at $x=[1:0:\cdots:0]$ and we may assume that the coefficient of $x_{n+1}^{n+1}$ is nonzero in $g$. Using the one parameter subgroup $[x_0:x_1:\cdots:x_n:x_{n+1}]\mapsto [x_0:t^{n+1} x_1:\cdots:t^{n+1} x_n:t^n x_{n+1}]$ ($t\in \bC^*$) of $\Aut(\bP^{n+1})$, it is not hard to see that $X$ specially degenerates to the K-polystable Fano variety $X_0$ as in Lemma \ref{lem:X_0}. Denote by $(\cX;\cL)$ this special test configuration of $X$.
Since $X_0$ is K-polystable, we know that $\Fut(\cX;\cL)=0$ (it can be viewed as the Futaki invariant of the induced $\bG_m$-action on $X_0$) and $X$ is K-semistable by \cite[Corollary 4]{BL18a}. 
As $X$ is not isomorphic to $X_0$, 
we conclude that $X$ is strictly K-semistable.

We next show that $\alpha(X)\ge \frac{n}{n+1}$. It then follows from \cite[Theorem 1.4]{OS-alpha} and the above consideration that $\alpha(X)=\frac{n}{n+1}$ and $X$ is strictly K-semistable. Let $V=(x_0=x_{n+1}=0)\subseteq X_0$ and let $D\sim_\bQ -K_X\sim H$ be an effective $\bQ$-divisor on $X$ (where $H$ is the hyperplane class). Assume that $D$ degenerates to $D_0\subseteq X_0$ under the above special degeneration of $X$ to $X_0$. Write $D_0=tV+D'$ where $V\not\subseteq \Supp(D')$. We now separate into two cases.

Assume first that $t\le 1$. We claim that $(X_0,\frac{n}{n+1}D_0)$ is lc and it then follows from inversion of adjunction (or the lower semicontinuity of lct) that $(X,\frac{n}{n+1}D)$ is also lc. As $D_0$ is $\bG_m$-invariant by construction, it suffices to check that $(X_0,\frac{n}{n+1}D_0)$ is lc at $x$ and along $V$ (any non-lc center is $\bG_m$-invariant and therefore either contains $x$ or is contained in $V$). Let $\pi:Y_0\rightarrow X_0$ be the weighted blowup as above that extracts the Koll\'ar component $F$ over $x$. Note that $-K_X\sim H$ and the linear system $|\pi^*H-(n+1)F|$ is base point free (its members include the strict transform of $(x_i=0)$ for $i=1,\cdots,n$), hence $\epsilon(F)\ge n+1$; on the other hand, since $(\pi^*H-(n+1)F)^n=0$, we also have $\epsilon(F)\le \tau(F)\le n+1$ by the inequality
\[\left((\pi^*H-(n+1)F)^{n-1}\cdot (\pi^*H-\tau(F)F)\right)\ge 0,\]
hence $\epsilon(F)=\tau(F)=n+1$. As $A_{X_0}(F)=n$, we see that $a(F;X_0,\frac{n}{n+1}D_0)\ge -1$, hence by the following Lemma \ref{lem:lct at wexc sing}, $(X_0,\frac{n}{n+1}D_0)$ is lc at $x$. Similarly, as $K_{X_0}+V+\frac{n}{n+1}D_0\sim_\bQ 0$, $K_V+\Delta_V+\frac{n}{n+1}D'|_V =  (K_{X_0}+V+\frac{n}{n+1}D')|_V$ is anti-nef (where $\Delta_V$ is the different), hence since $(V,\Delta_V)\cong (F,\Delta_F)$, the pair $(X_0,V+\frac{n}{n+1}D')$ is lc along $V$ by the exact same proof of Lemma \ref{lem:lct at wexc sing}. Therefore as $t\le 1$, $(X_0,\frac{n}{n+1}D_0)$ is also lc along $V$, proving our claim.

On the other hand, if $t>1$, then $D'\sim_\bQ cH$ for some $c\le \frac{n}{n+1}$ and by the argument above, $(X_0,D')$ is lc at $x$. Note that $x\not\in V$, it follows that $(X_0,D_0)$ is also lc at $x$ and by inversion of adjunction $(X,D)$ is lc at $x$ as well. If $Z$ is a positive dimensional non-lc center of $(X,D)$, then $Z$ is contained in the smooth locus of $X$ and we have $\mult_Z D>1$ by \cite[3.14.1]{Kol-sing-of-pairs}. But by \cite[Proposition 5]{Puk-mult-bound}, $\mult_Z D\le 1$, a contradiction. Thus the non-lc locus of $(X,D)$ is a union of isolated smooth points. By \cite[Theorem 1.6]{Z-cpi} (with $\Delta=L=0$), we immediately obtain $\lct(X;D)\ge \frac{n}{n+1}$. Therefore, in both cases $(X,\frac{n}{n+1}D)$ is lc. As $D$ is arbitrary, our proof is complete.
\end{proof}

The following lemma is used in the above proof.

\begin{lem} \label{lem:lct at wexc sing}
Let $(X,D)$ be pair and $x\in X$ a weakly exceptional singularity. Let $\pi:Y\rightarrow X$ be the plt blowup that extracts the unique Koll\'ar component $F$ over $x$. Assume that $a(F;X,D)\ge -1$, then $(X,D)$ is lc at $x$.
\end{lem}

\begin{proof}
By assumption, we have $K_Y+\Gamma+\lambda F=\pi^*(K_X+D)$ where $\lambda=-a(F;X,D)\le 1$ and $\Gamma$ is the strict transform of $D$. Let $\Delta_F$ be the different on $F$. Since $x\in X$ is weakly exceptional we have $\alpha(F,\Delta_F)\ge 1$. Therefore as $K_F+\Delta_F+\Gamma|_F=(K_Y+\Gamma+F)|_F\sim_\bQ (1-\lambda)F|_F$ is anti-nef, we see that $(F,\Delta_F+\Gamma|_F)$ is lc and then by inversion of adjunction $(Y,\Gamma+F)$ is also lc along $F$. Hence $(X,D)$ is lc at $x$.
\end{proof}

\subsection{K-unstable example}

Modifying the previous strictly K-semistable example, we can also construct, for each fix integer $e\ge 2$, K-unstable Fano varieties $Y$ of dimension $n\gg 0$ such that $\alpha(Y)=1-\frac{1}{n+2-e}$.

Let $e\ge 2$ and let $f(x_0,\cdots,x_n)$, $g(x_0,\cdots,x_{n+1})$ and $h(x_0,\cdots,x_n)$ be general homogeneous polynomials of degrees $e$, $n+2-e$ and $n+1-e$ respectively. Let $Y\subseteq\bP^{n+2}$ be the complete intersection defined by the equation $(f=g+x_{n+2}h=0)$. We now show that $Y$ is K-unstable with alpha invariant $\alpha(Y)=1-\frac{1}{n+2-e}$ in several steps. Let $x=[0:\cdots:0:1]\in \bP^{n+2}$.

\begin{lem} \label{lem:codim 2 wexc}
Assume that $n\ge 11+e$. Then with the above notations, $x\in Y$ is a weakly exceptional singularity.
\end{lem}

\begin{proof}
In local coordinates, $x\in Y$ is given by the equation $(f=g+h=0)\subseteq \bA^{n+2}_{x_0,\cdots,x_{n+1}}$. Let $\pi:W\rightarrow Y$ be the weighted blowup at $x$ with weights $(w_0,\cdots,w_{n+1})=(n+2-e,\cdots,n+2-e,n+1-e)$ and let $F$ be the exceptional divisor. Also let $\Delta_F$ be the different. Since $g$ is general, we may assume that the coefficient of $x_{n+1}^{n+2-e}$ is nonzero in $g$. We then have \[F=\left((f=h+x_{n+1}=0)\subseteq \bP(1^{n+1},n+1-e)\right)\cong \left((f=0)\subseteq\bP^n\right)\]
and $\Delta_F=(1-\frac{1}{n+2-e})S$ where $S=(f=h=0)\subseteq \bP^n$ and it suffices to prove that $\alpha(F,\Delta_F)\ge 1$. To see this, let $D\sim_\bQ -(K_F+\Delta_F)\sim_\bQ (1-\frac{1}{n+2-e})H$ be an effective divisor on $F$ (where $H$ is the hyperplane class). Similar to the proof of Corollary \ref{cor:weakly exceptional}, we may assume that $D$ is irreducible and not supported on $S$. Note that $S$ is a Calabi-Yau complete intersection of codimension $2$, hence by Lemma \ref{lem:general cpi P-regular}, Corollary \ref{cor:CY cpi} and our assumption that $n\ge 11+e$, we have $(S,D|_S)$ is lc. As $(F,D)$ is lc by \cite[Theorem 1.3]{C-hypersurface-lct}, we see that $(F,S+D)$ is also lc by inversion of adjunction. It follows that $(F,\Delta_F+D)$ is lc as well. Since $D$ is arbitrary, we obtain $\alpha(F,\Delta_F)\ge 1$.
\end{proof}

\begin{prop} \label{prop:k-unstable}
Notation as above. Then $Y$ is not K-semistable and $\alpha(Y)=1-\frac{1}{n+2-e}$ when $n\ge 10+e^2$.
\end{prop}

\begin{proof}
Let $\pi:W\rightarrow Y$ be the weighted blowup that extracts the Koll\'ar component $F$ over $x$ as in the previous proof. It is not hard to see that the linear system $|\pi^*(-K_Y)-(n+2-e)F|$ is base point free (consider strict transforms of the hyperplanes $(x_i=0)$ for $0\le i\le n$) and $(\pi^*H-(n+2-e)F)^n)=0$, hence $\epsilon(F)=\tau(F)=n+2-e$ by a similar argument as in Propisition \ref{prop:strict K-ss}. We then have $\vol_Y(-\pi^*K_Y-tF)=((-K_Y)^n)+(-t)^n(F^n)$ for all $0\le t\le \tau(F)$ and a straightforward computation yields
\begin{eqnarray*}
    \beta(F) & = & A_Y(F)-\frac{1}{((-K_Y)^n)}\int_0^{\tau(F)} \vol(-\pi^*K_Y-tF)    \mathrm{d}t \\
     & = & (n+1-e)-\frac{n}{n+1}\cdot (n+2-e) < 0
\end{eqnarray*}
as $e\ge 2$, hence $Y$ is not K-semistable by \cite[Corollary 1.5]{Fujita-valuative-criterion} or \cite[Theorem 3.7]{Li-minimize}.

Let $\alpha=1-\frac{1}{n+2-e}$ and assume that $n\ge 10+e^2>11+e$. We now show that $\alpha(Y)=\alpha$. Let $F$ be the Koll\'ar component over $x$. From the previous proof we have $\tau(F)=n+2-e$ and $A(F)=n+1-e$, hence $\alpha(Y)\le \alpha$. It remains to show that $(Y,\alpha D)$ is lc for any effective divisor $D\sim_\bQ -K_Y$. For this we may assume that $D$ doesn't contain the hyperplane section $Z=(x_{n+2}=0)\cap Y=(f=g=0)\subseteq \bP^{n+1}$ in its support (since $(Y,Z)$ is lc and being lc is preserved under convex linear combination). Since $f$ and $g$ are general and $n\ge 10+e^2$, by Lemma \ref{lem:cone cpi P-regular} and Corollary \ref{cor:codim 2 CY} we have $\lct(Z;|H|_\bQ)=1$ and in particular, $(Z,D|_Z)$ ic lc. By inversion of adjunction, $(Y,Z+D)$ is lc in a neighbourhood of $Z$ and hence the non-lc center of $(Y,D)$ is a finite union of isolated point. Therefore by \cite[Theorem 1.6]{Z-cpi} (with $\Delta=L=0$) we see that $(Y,\frac{n}{n+1}D)$ is lc over the smooth locus of $Y$. Note that $\alpha<\frac{n}{n+1}$, so it remains to check that $(Y,\alpha D)$ is lc at $x$. As $A_{Y,\alpha D}(Y)\ge A_Y(F)-\alpha\cdot \tau(F)=0$, this follows from Lemma \ref{lem:codim 2 wexc} and Lemma \ref{lem:lct at wexc sing} and we are done.
\end{proof}

\begin{rem}
Here we describe a K-polystable replacement of $Y$.
Assume $n\geq 20$.  Let $\pi:\cX\to C$ be a family of codimension two complete intersections of $\bP^{n+2}$ of degree $(e, n+2-e)$ over a smooth pointed curve $0\in C$, such that the special fiber $\cX_0\cong Y$ and $\pi$ is smooth over $C\setminus\{0\}$. Then by \cite[Theorem 1.3]{Z-cpi} we know that $\cX_t$ is a K-stable Fano manifold for any $t\in C\setminus\{0\}$. Thus by \cite{DS14, CDS15, Tia15}, after a finite base change $(0'\in C')\to (0\in C)$ if necessary, there exists a $\bQ$-Gorenstein flat projective family $\cX'_{C'}\to C'$ such that $\cX'_{0'}$ is a K-polystable $\bQ$-Fano variety and $\cX'_{C'}\setminus \cX'_{0'}\cong \cX_{C'}\setminus \cX_{0'}$ where $\cX_{C'}:=\cX\times_{C}C'$. Since $\cX_{0'}\cong 
Y$ is K-unstable, we know that $\cX'_{0'}$ is not isomorphic to $\cX_{0'}$. Hence Proposition \ref{prop:alpha sum>1} and Proposition \ref{prop:k-unstable} implies that $\alpha(\cX'_{0'})\leq 1-\alpha(\cX_{0'})=\frac{1}{n+2-e}$. In particular, when $e=2$ we obtain a K-polystable $\bQ$-Fano variety $\cX'_{0'}$ satisfying $\alpha(\cX'_{0'})\leq \frac{1}{n}$.
\end{rem}

\begin{proof}[Proof of Theorem \ref{thm:main}]
The existence of $X$ follows from Proposition \ref{prop:strict K-ss} and the existence of $Y$ follows from Proposition \ref{prop:k-unstable} with $e=2$. The fact that $X$ and $Y$ are both Gorenstein canonical and have a unique singular point follows from a straightforward computation.
\end{proof}

We conclude with the following question (c.f. \cite[Conjecture 1.6]{Jia17}). 

\begin{que}
Let $n=\dim X\geq 2$ be an integer. 
\begin{enumerate}
    \item Does there exist a K-semistable $\bQ$-Fano variety $X$ such that $\frac{1}{n+1}<\alpha(X)<\frac{1}{n}$?
    \item Does there exist a K-unstable $\bQ$-Fano variety $X$ such that $\frac{n-1}{n}<\alpha(X)<\frac{n}{n+1}$?
\end{enumerate}
\end{que}

\bibliography{ref}
\bibliographystyle{alpha}

\end{document}